\documentclass{article}

\usepackage[margin=3cm]{geometry}
\usepackage{amsmath}
\usepackage{amssymb}
\usepackage{amsthm}

\usepackage{pgf, tikz}
\usetikzlibrary{arrows, automata, positioning}

\usepackage{color}
\usepackage{hyperref}
\usepackage{subcaption}

\usepackage[version=3]{mhchem}
	\renewcommand{\longleftrightarrow}{\ce{<=>}}

\DeclareMathOperator{\interior}{int}
\DeclareMathOperator{\kernel}{ker}

\DeclareMathOperator{\sign}{sign}
\DeclareMathOperator{\supp}{supp}

\DeclareRobustCommand{\mybox}{\hfill $\Box$}
\DeclareRobustCommand {\mydiamond}{\hfill $\Diamond$}

\DeclareRobustCommand{\calC}{{\cal C}}
\DeclareRobustCommand{\calE}{{\cal E}}

\DeclareRobustCommand{\calO}{{\cal O}}
\DeclareRobustCommand{\calR}{{\cal R}}
\DeclareRobustCommand{\calS}{{\cal S}}

\newcommand{\Gnac}{G1}
\newcommand{\Gmon}{G2}
\newcommand{\Nc}{G3}
\newcommand{\consistency}{G4}
\newcommand{\rone}{r1}
\newcommand{\rtwo}{r2}
\newcommand{\rthree}{r3}

\newcommand{\Hzero}{\Gmon}

\newcommand{\foralpha}{\alpha}

\newcommand{\forcalR}{{\cal R}}

\makeatletter
\def\Ddots{\mathinner{\mkern1mu\raise\p@
\vbox{\kern7\p@\hbox{.}}\mkern2mu
\raise4\p@\hbox{.}\mkern2mu\raise7\p@\hbox{.}\mkern1mu}}
\makeatother

\DeclareRobustCommand{\r}{{\mathbb R}}
    \DeclareRobustCommand{\rplus}{\r_{\geqslant 0}}
    \DeclareRobustCommand{\rpluss}{\r_{> 0}}
\DeclareRobustCommand{\z}{{\mathbb Z}}

\DeclareRobustCommand{\zpluss}{\z_{> 0}}

\DeclareRobustCommand{\comment}[1]{}

\theoremstyle{plain}
\newtheorem{proposition}{Proposition}
\newtheorem{lemma}{Lemma}
\newtheorem{theorem}{Theorem}
\newtheorem{corollary}{Corollary}

\theoremstyle{definition}

\theoremstyle{remark}
\newtheorem{remark}{Remark}
\newtheorem{example}{Example}

\title{Intermediates and \\ Generic Convergence to Equilibria}
\author{Michael Marcondes de Freitas, Carsten Wiuf, and Elisenda Feliu \\ {\small Department of Mathematical Sciences, University of Copenhagen}}
\date{\today}

\begin{document}

\maketitle

\begin{abstract}
	Known graphical conditions for the generic and global convergence to equilibria of the dynamical system arising from a reaction network are shown to be invariant under the so-called successive removal of intermediates, a systematic procedure to simplify the network, making the graphical conditions  {considerably} easier to check.
	\end{abstract}

\vskip 2ex


\tableofcontents

\section{Introduction}\label{sec:intro}

In recent years many works in reaction network theory have been concerned with the idea of model reduction or simplification. This interest is expressed along various lines of investigation. One direction is the natural problem of providing simpler models to describe or explain the same biochemical phenomenon \cite{radulescu--gorban--zinovyev--noel-2012}. Another dimension is the consolidation of known model simplification techniques typically justified and applied {\em ad hoc}, such as quasi-steady state approximations \cite{king--altman-1956,cornish-bowden--2004}, into a formal procedure \cite{saez--wiuf--feliu-2015}. A third line of inquiry contemplates whether certain qualitative properties of a reaction network, for instance, number of steady states \cite{feliu--wiuf-2013} or the property of persistence \cite{mmf--feliu--wiuf-2015}, are invariant under a simplification procedure. This work fits within this last category. The qualitative property of interest is generic convergence to equilibria ---the property that almost every solution within each stoichiometric compatibility class  {approaches} the set of equilibria--- and the model simplification procedure is the successive removal of intermediates.

To illustrate our contribution, consider the one-site phosphorylation mechanism modeled by the reaction network
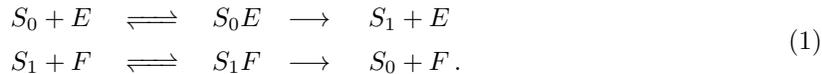
\begin{equation}\label{eq:one_site}
	\begin{array}{ccccc}
		S_0 + E &\longleftrightarrow& S_0E &\longrightarrow& S_1 + E \\[1ex]
		S_1 + F &\longleftrightarrow& S_1F &\longrightarrow& \phantom{\,.}S_0 + F\,. \\
	\end{array}
\end{equation}
In this mechanism, $S_0$ and $S_1$ are, respectively, the dephosphorylated and phosphorylated forms of some protein. The phosphorylation and dephosphorylation reactions are catalyzed by a kinase $E$ and a phosphatase $F$. Intermediate steps in the process during which protein and catalyst are bound to one another are captured in $S_0E$ and $S_1F$. Activation/deactivation motifs such as this one appear in many important intracellular signaling processes regulating cell proliferation, differentiation and apoptosis in eukaryotes ranging from yeast to mammals \cite{widmann--spencer--jarpe--johnson-1999}.

In \cite{angeli--deleenheer--sontag-2010}, sufficient graphical conditions for a reaction network to exhibit generic {and global} convergence to equilibria (within each stoichiometric compatibility class) were given. The technique consists of checking that the R-graph of the network is such that every simple loop has an even number of negative edges, a property known as the positive loop property, and that there exists a directed path between any two reaction vertices in the directed SR-graph. 
 {We show in Theorem~\ref{thm:dsrgraph}  that, under some additional assumptions that also are required for the result on generic and global convergence to hold, the existence of a directed path between any two reaction vertices in the directed SR-graph is equivalent to the simpler condition of  connectedness of the R-graph. Therefore the graphical conditions are reduced to two conditions on the R-graph.}

For the one-site phosphorylation mechanism above, the 
R-graph is displayed in Figure \ref{fig:one_site_r}, and one can readily see that it satisfies the aforementioned conditions.

\begin{figure}[!ht]
  \hfill
  \begin{subfigure}[b]{.45\textwidth}
        \centering
        \begin{tikzpicture}[
			roundnode/.style={},
			squarednode/.style={rectangle, draw=black, minimum size=5mm},
			> = stealth, 
	        shorten > = 1pt, 
    	    shorten < = 1pt, 
        	auto,
	   	    node distance = 1.5cm, 
    	   	semithick 
			]
			\node[squarednode]		(R1)								{\scriptsize $R_1$};
			\node[squarednode]		(R2)			[right=of R1]		{\scriptsize $R_2$};
			\node[squarednode]		(R3)			[below=of R2]		{\scriptsize $R_3$};
			\node[squarednode]		(R4)			[left=of R3]			{\scriptsize $R_4$};
 
			\draw[-]	 	(R1) edge node {\scriptsize $+$} (R2);
			\draw[-]	 	(R2) edge node {\scriptsize $+$} (R3);
			\draw[-]	 	(R3) edge node {\scriptsize $+$} (R4);
			\draw[-]	 	(R4) edge node {\scriptsize $+$} (R1);			
		\end{tikzpicture}
\caption{R-graph of \eqref{eq:one_site}}\label{fig:one_site_r}
  \end{subfigure}%
\hfill
\begin{subfigure}[b]{0.45\textwidth}		
	\centering
	\begin{tikzpicture}[
		roundnode/.style={},
		squarednode/.style={rectangle, draw=black, minimum size=5mm},
		> = stealth, 
        shorten > = 1pt, 
        shorten < = 1pt, 
        auto,
   	    node distance = 1.5cm, 
       	semithick 
		]
		\node[squarednode]		(R1)		at		(0, 0)		{\scriptsize $R^*_1$};
		\node[squarednode]		(R2)		at		(2, 0)		{\scriptsize $R^*_2$};
		 
		\draw[-]				(R1) edge node {\scriptsize $+$} (R2);			
	\end{tikzpicture}
	
	\vspace{1.2cm}
	\caption{R-graph of \eqref{eq:reduced_one_site}}	\label{fig:r_of_reduced_one_site_2}
\end{subfigure}
\hfill
    \caption{The     R-graphs of \eqref{eq:one_site} and of \eqref{eq:reduced_one_site}.}\label{fig:one_site}
\end{figure}
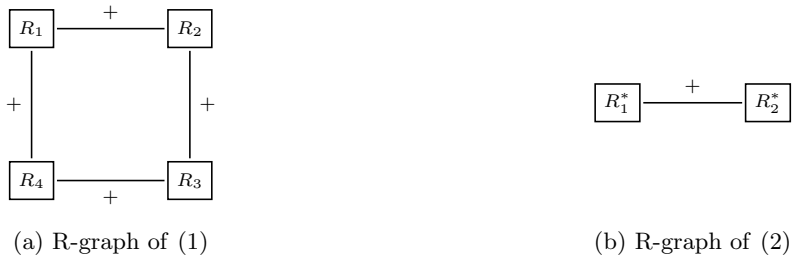

By successively removing the intermediates $S_0E$ and $S_1F$, what basically consists of ``collapsing'' the reaction paths through them, followed by canceling out the ``catalysts'' $E$ and $F$, which appear on both sides of their respective emerging reactions with the same stoichiometric coefficient, we obtain the simplified substrate network
\begin{equation}\label{eq:reduced_one_site}
	R^*_1\colon\quad S_0 \longrightarrow S_1
	\quad\quad\quad
	R^*_2\colon\quad S_1 \longrightarrow S_0\,.
\end{equation}
For this simplified network, the 
 R-graph is much simpler (Figure \ref{fig:r_of_reduced_one_site_2}), and  {therefore} the conditions for generic  {and global}  convergence are much easier to check. (The reason for not writing this as a single reversible reaction will become clearer when we introduce our working reaction network formalism in the next section.)

In what follows, we will show that 
 {connectivity and the positive loop property of the R-graph} are always invariant under the successive removal of intermediates  {under the assumptions of \cite{angeli--deleenheer--sontag-2010}}, meaning that the reduced network has them  {if and only if} the original one does also, as illustrated in the example above. Thus, the conditions for the original model can be checked in the often times much simpler reduced model. Therefore, although this ``invariance under reduction'' feature might be useful in the context of finding simpler models to describe the same observed phenomenon, it is also interesting on its own as a mathematical tool to analyze large, complicated models, even if the network obtained through the reduction procedure might not necessarily be understood to be biologically meaningful.

The approach to generic  {and global} convergence to equilibria in \cite{angeli--deleenheer--sontag-2010} is based upon the monotone systems theory of M. W. Hirsch \cite{hirsch-1988,smith-1995,hirsch--smith-2005}. The reader familiar with that theory will likely notice the connection, although most of the details have been deliberately hidden in our presentation by framing all concepts and results pertaining to monotonicity directly in terms of the graphical conditions given in \cite{angeli--deleenheer--sontag-2010}.

This paper is organized as follows. In Section \ref{sec:crns} we introduce our basic notation and working definition of reaction network, then review the graphical conditions for  {generic and global} convergence to equilibria of \cite{angeli--deleenheer--sontag-2010}  {and state Theorem~\ref{thm:dsrgraph}}. In Section \ref{sec:intermediates} we describe a systematic procedure to obtain a reduced reaction network by successively removing intermediates from a given network. We then state our main results concerning the invariance of the aforementioned conditions for  {generic and global} convergence under this procedure, and apply them to several examples in the recent reaction network literature. The last section is devoted to the technical details of the proofs of our main results.

\section{Reaction Networks}\label{sec:crns}

In what follows, we denote the set of nonnegative real numbers by $\rplus$, and denote the set of strictly positive real (respectively, integer) numbers by $\rpluss$ (respectively, $\zpluss$). Given $n \in \zpluss$, we write $[n] := \{1, \ldots, n\}$. By convention, $[0] := \varnothing$. For each $a \in \r$,
			\[
				\sign a :=
					\left\{
					\begin{array}{rl}
						1\,, &\text{if} \ a > 0 \\[1ex]
						0\,, &\text{if} \ a = 0 \\[1ex]
						-1\,, &\text{if} \ a < 0\,.	
					\end{array}
					\right.
			\]

\subsection{Basic Formalism}

We start by introducing our working definition of reaction network. A {\em complex} over a nonempty, finite set $\calS = \{S_1, \ldots, S_n\}$ is a vector $(\alpha_1, \ldots, \alpha_n) \in \rplus^n$, often also expressed as the formal linear combination $\alpha_1S_1 + \cdots + \alpha_nS_n$. In this context, the elements of $\calS$ are referred to as the {\em species} constituting the complex,  {and $\alpha_i$ as the \emph{stoichiometric coefficient} of $S_i$ in the complex}. A {\em reaction} over a set of complexes $\calC$ is an object of the form $y \longrightarrow y'$ or $y \longleftrightarrow y'$ for some $y, y' \in \calC$, $y \neq y'$. The former are referred to as {\em irreversible} reactions, while the latter are called {\em reversible}. In either case, $y$ is called the {\em reactant} of the reaction, and $y'$ the {\em product}.

A {\em reaction network}  {(or just network)} is an ordered triple $G = (\calS, \calC, \calR)$ where $\calC$ is a set of complexes over a nonempty, finite set of species $\calS = \{S_1, \ldots, S_n\}$, and $\calR = \{R_1, \ldots, R_m\}$ is a nonempty, finite set of reactions over $\calC$.   {We do not assume that the reactions  necessarily all are different  nor that reactions of the form   $y\cee{<=>}y'$, $y'\cee{<=>}y$, $y\cee{->}y'$ and $y'\cee{->}y$ are mutually exclusive. }
  {The reason for this (unusual) convention will be made clear later on and is essential for our results to hold.}

We write $\calR = \calR_{\rightarrow} \cup \calR_{\leftrightarrow}$, where $\calR_{\rightarrow}$ and $\calR_{\leftrightarrow}$ are the disjoint subsets of irreversible and reversible reactions, respectively. We further assume that, for each $i \in [n]$, there exists an $(\alpha_1, \ldots, \alpha_n) \in \calC$ such that $\alpha_i > 0$, and, for each $y \in \calC$, there exists a reaction in $\calR$ having $y$ as a reactant or product; in other words, $\calS$ (respectively, $\calC$) is the minimal set over which $\calC$ (respectively, $\calR$) may be defined. 

For each $j \in [m]$, let $\alpha_{1j}S_1 + \cdots + \alpha_{nj}S_n$ be the reactant and $\alpha'_{1j}S_1 + \cdots + \alpha'_{nj}S_j$ be the product of reaction $R_j$. With this notation, we may define the $n \times m$ matrix $N$,
\[
	N_{ij} := \foralpha'_{ij} - \foralpha_{ij}\,,
	\quad
	i = 1, \ldots, n\,,
	\quad
	j = 1, \ldots, m\,,
\]
known as the {\em stoichiometric matrix} of the network. The column-space of $N$, which is a subset of $\r^n$, is called the {\em stoichiometric subspace of $G$}, and denoted by $\Gamma$. A vector $c \in \r^n$ is said to be a {\em conservation law} of $G$ if $c \in \Gamma^\perp$. The subsets $(s_0 + \Gamma) \cap \rplus^n$, for $s_0 \in \rplus^n$, are known as the {\em stoichiometric compatibility classes} of $G$.

The system of ordinary differential equations modeling the evolution of the concentrations of the species of the network $G$ is then given by
\begin{equation}\label{eq:crnode}
\frac{ds}{dt} = Nr(s(t))\,,
\quad
t \in \rplus\,,
\quad
s \in \rplus^n\,,
\end{equation}
where $r = (r_1, \ldots, r_m)\colon \rplus^n \rightarrow \r^m$ is a vector-valued function  {that models the kinetic rate of each reaction as a function of the species concentrations.  If two reactions in $\calR$ are the same, they might or might not have the same kinetic rates.}

 {Our results are applicable to networks and kinetic rates that fulfil the requirements for generic and global convergence  in \cite{angeli--deleenheer--sontag-2010}, which we review in the next subsection.
The assumptions (\Gnac)--(\consistency) below refer to properties of the network and the assumptions (\rone)--(\rthree) refer to properties of the kinetic rates.
}

\begin{itemize}
	\item[(\Gnac)] There are no {\em auto-catalytic} reactions, meaning that no species can appear as both reactant and product in any reaction. Thus, $\alpha_{ij}\alpha'_{ij} = 0$ for any reaction $R_j \in \calR$ and any species $S_i \in \calS$.
	
	\item[(\Gmon)] Each species in $\calS$ takes part in at most two reactions in $\forcalR$.
	
	\item[(\Nc)] The network is {\em conservative}, that is, it has a conservation law $c \in \rpluss^n$  {with all entries positive}.
	
	 {
	\item[(\consistency)] The network is {\em consistent}, that is,  there exists a vector $v \in \ker N$ such that $v_j>0$ for all $j\in [m]$  for which  $R_j$ is irreversible.}

\medskip
	\item[(\rone)] For each $j \in [m]$, if $R_j$ is irreversible, then $r_j(s) \geqslant 0$, $s \in \rplus^n$; if $R_j$ is reversible, then $r_j = r_j^f - r_j^b$, where $r_j^f(s), r_j^b(s) \geqslant 0$, $s \in \rplus^n$. Furthermore, all the $r_j, r_j^f, r_j^b\colon \rplus^n \longrightarrow \rplus$ have continuously differentiable extensions to a neighborhood $\calO$ of $\r_{\geqslant 0}^{n}$.

	\item[(\rtwo)] For each $j \in [m]$, and for each $s = (s_1, \ldots, s_{n}) \in \r_{\geqslant 0}^{n}$,
		\begin{itemize}
			\item[({\em i}\,)] if $R_j$ is irreversible, then
				\[
					s_k = 0
					\
					\text{for some} \ k \in \{i \in [n]\,|\ \alpha_{ij} > 0\}
					\quad
					\Rightarrow
					\quad
					r_j(s) = 0\,;
				\]
			
			\item[({\em ii}\,)] if $R_j$ is reversible, then 
				\[
					s_k = 0
					\
					\text{for some} \ k \in \{i \in [n]\,|\ \alpha_{ij} > 0\}
					\quad
					\Rightarrow
					\quad
					 { r_j^f(s) = 0\,,}
				\]
				and
				\[
					s_k = 0
					\
					\text{for some} \ k \in \{i \in [n]\,|\ \alpha'_{ij} > 0\}
					\quad
					\Rightarrow
					\quad
					 { r_j^b(s) = 0\,.}
				\]
		\end{itemize}

	\item[(\rthree)] For each $j \in [m]$,
		\begin{itemize}
			\item[({\em i}\,)] if $R_j$ is irreversible, then
				\begin{equation*}
					\frac{\partial r_j}{\partial s_i}(s)
					\left\{
						\begin{array}{rl}
						\geqslant 0\,, & \text{if} \ \alpha_{ij} > 0 \\[1ex]
						= 0\,, & \text{if} \ \alpha_{ij} = 0\,.
						\end{array}
					\right.
				\end{equation*}
			
			\item[({\em ii}\,)] if $R_j$ is reversible, then
				\begin{equation*}
					\frac{\partial r_j}{\partial s_i}(s)
					\left\{
						\begin{array}{rl}
						\geqslant 0\,, & \text{if} \ \alpha_{ij} > 0 \\[1ex]
						= 0\,, & \text{if} \ \alpha_{ij} = 0 \\[1ex]
						\leqslant 0\,, & \text{if} \ \alpha'_{ij} > 0\,.
						\end{array}
					\right.
				\end{equation*}
		\end{itemize}
		Furthermore, the inequalities are strict in $\rpluss^n$.
\end{itemize}

 {  
The hypotheses (\Gnac)--(\consistency)  substantially reduce the number of reaction networks under consideration in \cite{angeli--deleenheer--sontag-2010} as well as in this study. While the hypotheses (\Gnac), (\Nc) and (\consistency) are fulfilled for many realistic networks, the hypothesis (\Gmon) is very restrictive (and also limits the number of versions of the same reaction that can be  in $\calR$). Nevertheless, there are several relevant and arbitrarily large networks that fall into our setting. Some examples are given in Subsection~\ref{subsec:examples}. 
}

\begin{remark}\label{rem:collapse}
In the literature, one typically defines reaction networks directly from their {\em reaction graphs} \cite{mmf--feliu--wiuf-2015}, keeping reciprocal reactions as distinct reactions, or, alternatively, by collapsing each pair of reciprocal reactions into a single reversible reaction \cite{angeli--deleenheer--sontag-2010}. Our approach accommodates both extremes, plus anything in between, since it does not preclude the possibility that both $y \longrightarrow y', y' \longrightarrow y \in \calR$. In other words, one  {may} choose at the beginning which pairs of reciprocal reactions to collapse into a single reversible reaction, and which ones not to,  { as long as (\Gmon) is not violated}.

 {We note that our setting is slightly more general than that of \cite{angeli--deleenheer--sontag-2010}. Specifically, we have some freedom in choosing which reactions are treated as reversible and which as two irreversible reactions as mentioned above. Also  our assumptions on the kinetic rates are slightly less restrictive. However, the results  we make use of from \cite{angeli--deleenheer--sontag-2010}, also hold in our setting.}

 { 
The hypotheses (\rone)--(\rthree) on the kinetic rates impose standard restrictions on the rate of each reaction, when reversible reactions are considered as two distinct reactions (a forward reaction with kinetic rate $r_j^f$ and a backward reaction with kinetic rate $r_j^b$). Under this consideration, assumption (\rone) states the nonnegativity of the rates, assumption (\rtwo) requires that the rates vanish  if one of the species in the reactant is not present, and assumption (\rthree), together with (\Gnac), gives that the kinetic rates increase in the concentrations of the species in the reactant and do not depend on any other concentration. }

Hypotheses (\rone)--(\rthree) are satisfied under the most common kinetic assumptions in the literature, namely, mass-action, or more general power-law kinetics, Michaelis-Menten kinetics, or Hill kinetics, as well as combinations of these \cite[pages 585--586]{angeli--deleenheer--sontag-2010}. 

It follows from (\rtwo) and \cite[Theorem 5.6]{smirnov-2002} that $\r_{\geqslant 0}^{n}$ is forward invariant for the flow of (\ref{eq:crnode}). We then conclude that the interior, $\rpluss^n$, is also forward-invariant via \cite[Remark 16.3(h)]{amann-1990}. (See also \cite[Section VII]{Sontag:2001}.) And in view of (\Nc), the trajectories of (\ref{eq:crnode}) are defined for all positive time, and also precompact.  {If the reverse implications in (\rtwo) hold, then (\consistency) is required for system \eqref{eq:crnode} to admit positive equilibria. }
\mybox
\end{remark}

\subsection{Graphical Conditions for  {Generic and Global} Convergence}

We now review the concepts and results from \cite{angeli--deleenheer--sontag-2010} that we will need, pointing out that they still hold in our slightly more general setting.  {We state also Theorem~\ref{thm:dsrgraph}, that simplifies the graphical conditions to be checked for generic and global convergence.}  {We start by introducing three graphical objects and their properties. }

The {\em directed SR-graph} of a reaction network $G$ is the directed, bipartite, labeled graph $G^\rightarrow_{SR} = (V^\rightarrow_{SR}, E^\rightarrow_{SR}, L^\rightarrow_{SR})$ defined as follows. The set of vertices $V^\rightarrow_{SR}$ is the disjoint union  {of the set of species and the set of reactions,}
\[
	V^\rightarrow_{SR} := \calS \cup \forcalR = \calS \cup (\forcalR_\rightarrow \cup \forcalR_\leftrightarrow)\,.
\]
The set of edges $E_{SR}^\rightarrow$ and the labeling $L_{SR}^\rightarrow$ are then characterized as follows.
\begin{itemize}
\item[({\em i}\,)] An ordered pair $(S_i, R_j) \in \calS \times \forcalR_\rightarrow$ belongs to $E_{SR}^\rightarrow$  {if and only if} $S_i$ is a reactant of $R_j$, that is,  {if and only if} $\foralpha_{ij} > 0$.

\item[({\em ii}\,)] An ordered pair $(S_i, R_j) \in \calS \times \forcalR_\leftrightarrow$ belongs to $E_{SR}^\rightarrow$  { if and only if} $S_i$ appears on either side of $R_j$, that is,  {if and only if} $\foralpha_{ij} + \foralpha'_{ij} > 0$.

\item[({\em iii}\,)] An ordered pair $(R_j, S_i) \in \forcalR \times \calS$ belongs to $E_{SR}^\rightarrow$  {if and only if} $S_i$ is part of $R_j$ as either a reactant or a product, that is,  {if and only if} $\alpha_{ij} + \alpha'_{ij} > 0$.

\item[({\em iv}\,)] $L_{SR}^\rightarrow(S_i, R_j) := - \sign N_{ij}$ for every $(S_i, R_j) \in E_{SR}^\rightarrow$, and $L_{SR}^\rightarrow(R_j, S_i) := - \sign N_{ij}$ for every $(R_j, S_i) \in E_{SR}^\rightarrow$.
\end{itemize}

 {In plain words, an edge from a reaction vertex to a species vertex indicates that the species is part of the reaction. An edge from a species vertex to a reaction vertex indicates that the species is in the reactant, if the reaction is irreversible, or that the species is part of the reaction, if the reaction is reversible. This type of edge encodes that the kinetic rate of the reaction depends on the concentration of the species (refer to (r3)). The label of the edge is then simply given by minus the entry of the stoichiometric matrix corresponding to the given species and reaction. }

The directed SR-graph is said to be {\em R-strongly connected} if, for every $R_j, R_k \in \forcalR$, there exists a directed path in $G^\rightarrow_{SR}$ connecting $R_j$ to $R_k$.

\begin{remark}\label{rem:SR_labeling}
	If $(S_i, R_j), (R_j, S_i) \in E^{\rightarrow}_{SR}$ for some $i \in [n]$ and some $j \in [m]$, then both edges get the same label. 
	 {The  only edges $(U,V)\in E^{\rightarrow}_{SR}$ for which $(V,U)\notin E^{\rightarrow}_{SR}$  are of the form $(R_j, S_i)$ with $R_j$ irreversible and $S_i$ in the product of $R_j$.  
}
	\mybox
\end{remark}

The {\em SR-graph} of $G$ is the undirected, labeled graph $G_{SR} = (V_{SR}, E_{SR}, L_{SR})$ where
	\[
		V_{SR} := V^\rightarrow_{SR} = \calS \cup \forcalR\,,
	\]
	\[
		E_{SR} := \{\{S_i, R_j\}\,|\ (S_i, R_j) \in E^\rightarrow_{SR} \ \text{or} \ (R_j, S_i) \in E^\rightarrow_{SR}\} = \{ \{S_i, R_j\}\,|\ N_{ij} \neq 0\}\,,
	\]
	and
	\[
		L_{SR}(\{S_i, R_j\}) := - \sign N_{ij}\,,
		\quad
		\{S_i, R_j\} \in E_{SR}\,.
	\]
In view of Remark \ref{rem:SR_labeling}, the SR-graph is simply the undirected graph underlying the directed SR-graph, and there are no multiple edges connecting any two vertices.

The {\em R-graph} is the undirected, labeled graph $G_R = (V_R, E_R, L_R)$ 
 {with vertices given by the set of reactions and such that two reactions are connected by an edge if there exists a species that is part of both reactions. Specifically, the R-graph is constructed as follows.}
The vertices set is defined as 
	\[
		V_R := \forcalR\,.
	\] 
	Furthermore,
	\[
		E_R := \{\{R_j, R_k\}\,|\ j \neq k \ \text{and} \ N_{ij}N_{ik} \neq 0 \ \text{for some} \ i \in [n]\}\,,
	\]
	and, for each $\{R_j, R_k\} \in E_R$,  {the labeling is given by}
	\[
		L_R(\{R_j, R_k\}) := \{-\sign N_{ij}N_{ik}\,|\ N_{ij}N_{ik} \neq 0 \ \text{and} \ i \in [n]\}\,.
	\]
	We emphasize that $L_R$ is a set-valued function.  

The R-graph is said to have the {\em positive loop property} if every labeled simple loop
\[
	R_{j_1} \stackrel{L_1}{\text{---}} R_{j_2} \stackrel{L_2}{\text{---}} \cdots \stackrel{L_{\ell-1}}{\text{---}} R_{j_\ell} \stackrel{L_\ell}{\text{---}} R_{j_1}
\]
in $G_R$ has an even number of negative labels, that is, $L_1L_2 \cdots L_\ell = 1$ for any choice of $L_1 \in L_R(\{R_{j_1}, R_{j_2}\})$, $L_2 \in L_R(\{R_{j_2}, R_{j_3}\})$, \ldots, $L_\ell \in L_R(\{R_{j_\ell}, R_{j_1}\})$.

The R-graph can be obtained from the SR-graph by placing an edge between two reaction vertices in the R-graph whenever there is a length-2 path connecting the two corresponding reaction vertices in the SR-graph, and labeling that edge with the opposite of the product of the labels along the length-2 path in the SR-graph. If there are more than one length-2 path connecting any two reaction vertices, it is possible that the corresponding edge in the R-graph gets multiple labels. The positive loop property of the R-graph can then be checked by inspecting the SR-graph from which it is built. This was done in \cite{angeli--deleenheer--sontag-2010}, and we cite the relevant result here for ease of reference.

Let $\Lambda\colon V_0 \ \text{---} \ V_1 \ \text{---} \ \cdots \ \text{---} \ V_{2\lambda} \ = \ V_0$, $\lambda \in \zpluss$, be any simple loop in the SR-graph. If
	\[
		\prod_{k = 1}^{2\lambda} L_{SR}(\{V_{k-1}, V_k\}) = (-1)^\lambda\,,
	\]
	then we call $\Lambda$ an {\em e-loop}. Otherwise we call it an {\em o-loop}. In this definition, we know the length of a simple loop in the SR-graph is always an even number because the SR-graph is a bipartite graph.
 
\begin{proposition}\label{prop:posloop}
	Let $G$ be a reaction network satisfying {\em (G1)--(G2)}. Then the R-graph has the positive loop property  {if and only if} all simple loops in the SR-graph are e-loops.
\end{proposition}

\begin{proof}
	See \cite[Proposition 4.5]{angeli--deleenheer--sontag-2010}.	
	\mybox
\end{proof}

\begin{remark}\label{rem:orthant_cone}
When the R-graph has the positive loop property,  {any edge $\{R_j, R_k\}$ along a loop has only  one label. In other words, $L_R(\{R_j, R_k\})$ consists of one element,  and by abuse of notation we identify this set with its unique element.}   {W}e may  {then} associate an  {orthant}
	\[
		K = \{(x_1, \ldots, x_m) \in \r^m\,|\ \sigma_1x_1, \ldots, \sigma_mx_m \geqslant 0\}
	\]
	with the network by defining the {\em sign pattern} $\sigma = (\sigma_1, \ldots, \sigma_m) \in \{\pm 1\}^m$ as follows. First suppose the R-graph is connected. Set $\sigma_1 := 1$. For  {$i \in [m] \backslash \{1\}$,}  consider any simple path $1 = i_0 \ \text{---} \ i_1 \ \text{---} \ \cdots \ \text{---} \ i_k = i$ joining $1$ and $i$, then set
	\begin{equation}\label{eq:sigma_def_b}
		\sigma_i
		:= 
		\prod_{d = 1}^{k} L_R \left( \{R_{i_{d-1}}, R_{i_{d}}\} \right) \,.
	\end{equation}
	In view of the positive loop property, this definition does not depend on the choice of the path. Indeed, the union of the edges of any two simple paths joining $1$ and $i$ is a union of simple loops   { with the edges common in both paths}. The product of the labels of the edges of the two paths is thus $1$, hence the products of the labels of the edges of each of the two paths agree.

If $G_R$ is not connected, then we apply the procedure to each connected component, starting by setting $\sigma_i := 1$ for the smallest index $i \in [m]$ such that $R_i$ belongs to that component.
\mybox
\end{remark}

In what follows, given a reaction network $G$ such that its R-graph has the positive loop property, we will always assume that $\sigma = (\sigma_1, \ldots, \sigma_m)$ is the sign pattern defined above, and $K$ the corresponding  {orthant}.  {We note that the orthant depends implicitly on $\calR$, that is, on how we choose to represent the reactions of the network. In particular, the dimensionality of the orthant equals the cardinality of $\calR$.}

\begin{proposition}\label{prop:p1_or_p2}
	Let $G$ be a reaction network satisfying {\em (G1)--(G2)}. Suppose that the R-graph has the positive loop property, and the directed SR-graph is R-strongly connected. Let $N$ be the stoichiometric matrix, and $K$ be the  {orthant} given by the construction in Remark \ref{rem:orthant_cone}. Then, either
	\begin{itemize}
		\item[{\em (P1)}] $\kernel N \cap K = \{0\}$\,,
	\end{itemize}
	or
	\begin{itemize}
		\item[{\em (P2)}] $\kernel N \cap \interior K \neq \varnothing$\,.
	\end{itemize}
\end{proposition}

\begin{proof}
	See \cite[Lemma 6.1]{angeli--deleenheer--sontag-2010}.
	\mybox
\end{proof}

 {In view of the following theorem, if (G1)-(G4) hold, then checking whether the directed SR-graph of $G$ is R-strongly connected 
is reduced to checking the much simpler condition of whether the R-graph is connected. }

 {
\begin{theorem}\label{thm:dsrgraph}
	Let $G$ be a reaction network satisfying
{\em (G1)--(G4)}. Then the following statements are equivalent:
\begin{itemize}
\item  the directed SR-graph of $G$ is R-strongly connected, 
\item the SR-graph of $G$ is connected,
\item  { the R-graph of $G$ is connected}. 
\end{itemize}
\end{theorem}

The proof of the theorem is given in Section \ref{sec:proofs}.}
Recall that the flow of (\ref{eq:crnode}) is said to be {\em bounded-persistent} if $\omega(s_0) \cap \partial \rplus^n = \varnothing$ for each $s_0 \in \rpluss^n$, where
\[
	\omega(s_0) := \bigcap_{\tau \geqslant 0} \overline{\bigcup_{t \geqslant \tau} \{\sigma(t, s_0)\}}
\]
is the {\em omega-limit set} of $s_0$.

\begin{proposition}\label{prop:ads-2010-thm2}
	Let $G$ be a reaction network satisfying {\em (G1)-- {(G4)}} and {\em (r1)--(r3)}. Suppose that the flow of {\em (\ref{eq:crnode})} is bounded-persistent. Suppose, in addition, that the R-graph has the positive loop property and  {is connected}. 
	Then,
	\begin{itemize}
		\item[{\em(}i\,{\em)}] if {\em (P1)} holds, then there exists a Lebesgue measure-zero $D \subseteq \rpluss^n$ such that all solutions of {\em (\ref{eq:crnode})} starting in $\rpluss^n \backslash D$ converge to the set of equilibria, and
		
		\item[{\em(}ii\,{\em)}] if {\em (P2)} holds, then all solutions of {\em (\ref{eq:crnode})} starting in $\rpluss^n$ converge to an equilibrium. Furthermore, this equilibrium is unique within each stoichiometric compatibility class.
	\end{itemize}	
\end{proposition}

\begin{proof}
	 {The result follows from \cite[Corollary 1 and Theorem 2]{angeli--deleenheer--sontag-2010} after applying  Theorem~\ref{thm:dsrgraph} to replace the   condition that the R-graph is connected by the R-strong connectivity of the directed SR-graph}. The  {orthant} constructed from the R-graph in Remark \ref{rem:orthant_cone} is the same as the  {orthant} given by \cite[Corollary 1]{angeli--deleenheer--sontag-2010}.
	
	\mybox	
\end{proof}

 {%
\begin{remark}
In view of Proposition \ref{prop:p1_or_p2} and Theorem~\ref{thm:dsrgraph}, Proposition \ref{prop:ads-2010-thm2} covers all possible scenarios because either (P1) or (P2) occur with the hypotheses of the proposition. 
 {Leaving the technicalities aside,  case (i) of the proposition states that almost all trajectories will approach an equilibrium point over time and hence cannot  escape to infinity (in any direction). However, there might be multiple equilibria even within a single stoichiometric compatibility class. Case (ii) states that there is a unique equilibrium in each of these classes and that all trajectories converge to it. Hence the conclusion in case (ii) is stronger than that in case (i).}
\mybox	
\end{remark}}

 {
\begin{remark} \label{rmk:RvsSR}
 By invoking     Proposition \ref{prop:posloop}  {and Theorem~\ref{thm:dsrgraph}}, we might replace the graphical conditions in  Proposition \ref{prop:ads-2010-thm2} with the conditions ``all simple loops in the SR-graph are e-loops" and the SR-graph is connected.  {The proofs of our reduction results are based on the invariance of the conditions on the SR-graph, which makes the proofs simpler. However, in specific examples, it might be simpler to check the conditions directly on the R-graph.}
\mybox	
\end{remark}}

 {
\begin{remark} \label{rk:G4}
	By \cite[Theorem 1]{angeli--deleenheer--sontag-2007c}, a conservative and bounded-persistent reaction network is consistent. This implies that assumption (G4) in Proposition \ref{prop:ads-2010-thm2} is redundant provided the other assumptions of the proposition hold. We have chosen to include it as an assumption, because our reduction results on the graphical properties require (G4) to hold, but do not require the flow to be bounded-persistent.
	\mybox	
\end{remark}}

\section{Main Results}\label{sec:intermediates}

This work is essentially about how the graphical conditions for  {generic and global} convergence to equilibria reviewed in Proposition \ref{prop:ads-2010-thm2} are invariant under the removal of so-called intermediates. However, the removal of intermediates in the sense they are typically defined in the reaction network literature \cite{feliu--wiuf-2013} often gives rise to auto-catalytic reactions, which are not allowed in our formalism because of (\Gnac). As we will see, if the problematic species were to appear only in the reactions emerging from the removal of intermediates, and with the same stoichiometric coefficients in both reactant and product sides, then they could be simply ``cancelled out,'' so that the network obtained from their removal still satisfies (\Gnac). 

We begin this section by giving a formal description of the procedure of removal of intermediates. We state our main results in Subsection \ref{subsec:main_result}, and discuss several examples from the literature in Subsection \ref{subsec:examples}. In Subsection \ref{subsec:exceptions}, we briefly contrast our working definition of intermediates with other variants in the literature, giving some examples and counterexamples motivating our choices in this context.

\subsection{Removal of Intermediates}\label{subsec:intermediates}

Let $G = (\calS, \calC, \calR)$ be a reaction network.
For each $y =  {(\alpha_1, \ldots, \alpha_n) \in \r^n}$, we  {define its support  as }
\[
	\supp y := \{S_i \in \calS\,|\ \alpha_i \neq 0\}\,.
\]

A species $Y \in \calS$ is  called an \emph{intermediate} of $G$ if the following two properties hold:

\begin{itemize}
	\item[(I1)] $Y \in \calC$, and $\supp y \cap \supp Y = \varnothing$ for every complex $y \in \calC \backslash \{Y\}$.
	
	\item[(I2)] There exist 
	$y = \alpha_1S_1 + \cdots + \alpha_nS_n$ and $y' = \alpha'_1S_1 + \cdots + \alpha'_nS_n$ in $\calC \backslash \{Y\}$, $y \neq y'$, such that
		\begin{itemize}
			\item[({\em i}\,)] either $y \longrightarrow Y$ or $y \longleftrightarrow Y$ is a reaction in $\calR$,
			
			\item[({\em ii}\,)] either $Y \longrightarrow y'$ or $Y \longleftrightarrow y'$ is a reaction in $\calR$,
			
			\item[({\em iii}\,)] $\displaystyle \sum_{S_i \in \calE} \alpha_iS_i = \sum_{S_i \in \calE} \alpha'_iS_i =: e$, where $\calE := \supp y \cap \supp y'$, and
		\end{itemize}
\end{itemize}

 {Condition (I1) states that $Y$ does not react with any other species, and the first two items of (I2) state that $Y$ is produced in  
 one reaction and consumed in 
 one reaction. Condition (I2)(iii) imposes that any species  appearing  in $y$ and $y'$ does so with the same stoichiometric coefficient. } 

 {When (G2) holds, then  reversible reactions in (I2)(i)-(ii) cannot be considered as two irreversible reactions and further, $y$ and $y'$ are uniquely determined. Additionally, if $\supp e$ is not empty, then  (\Hzero) implies that no species in $\supp e$ can take part in any other reaction in $G$ besides $y \ \text{---} \ Y$ and $Y \ \text{---} \ y'$, where the dash `---' is a placeholder for `$\longrightarrow$' or `$\longleftrightarrow$.' }

If (I1) and (I2) hold, then we may construct a reaction network $G^* = (\calS^*, \calC^*, \calR^*)$, called the {\em reduction of $G$ by the removal of the intermediate $Y$},  {by collapsing the reaction paths going through $Y$, and cancelling out any emerging ``catalysts.'' 
}
 {Although the construction of the reduced network can be defined in general \cite{feliu--wiuf-2013}, it becomes  simpler under assumption (G2). Since  we  require this assumption to hold in Theorem \ref{thm:main1} and \ref{thm:main2} below, we give the construction under this assumption. }

Specifically,   we define $\calR^* := \calR^*_c \cup \calR^*_Y$, where $\calR^*_c$ is identified with the subset of $\calR$ of reactions not involving the complex $Y$, and
	\[
		\calR^*_Y 
		:=
		\left\{
			\begin{array}{rl}
				\{y - e \longleftrightarrow y' - e\}\,, & \quad \text{if} \ y \longleftrightarrow Y, Y \longleftrightarrow y' \in \calR \\[1ex]
				\{y - e \longrightarrow y' - e\}\,, & \quad \text{if} \ y \longrightarrow Y \in \calR \ \text{or} \ Y \longrightarrow y' \in \calR \,.
			\end{array}
		\right.
	\]
	 {Here, `$-$' in $y - e$, and so on, denotes linear subtraction in $\r^n$.}
	We set $\calC^*$ to be the set of reactant and product complexes in the reactions in $\calR^*$, and $\calS^*$ is the set of species that are part of some complex in $\calC^*$. In the above description, we think of the reactant and product sides of a reaction in $\calR$ or $\calR^*$ as the formal linear combinations of participating species. Note that $\supp e$ might be empty and  {that  no species in $\supp e$ is present in $G^*$.}

 {
For example, in the one-site phosphorylation mechanism (\ref{eq:one_site}) from the introduction, $S_0E$ is an intermediate. It can be removed as follows. The  reaction path through $S_0E$ is $S_0 + E \longrightarrow S_0E \longrightarrow S_1 + E$. We first collapse the path into $S_0 + E \longrightarrow S_1 + E$, and then cancel out the emerging catalyst $E$. This yields  the reaction $S_0 \longrightarrow S_1$; see (\ref{eq:reduced_one_site}).}

 { Since $Y$ only appears in the two reactions involving  $y$ and $y'$, the removal of $Y$ does not affect the other reactions in the network. }  {The contracted reaction might already exist in $\calR^*_c$, or one or both of the irreversible reactions might likewise be in $\calR^*_c$.  For example, the one-site phosphorylation mechanism with reversible reactions $S_0+E\longleftrightarrow S_0E\longleftrightarrow S_1+E$, coupled with spontaneous dephosphorylation $S_1\longrightarrow S_0$, leads to the network $S_0\longleftrightarrow S_1$, $S_1\longrightarrow S_0$, upon reduction by $S_0E$. }

\begin{remark}\label{rem:seven}
 {The removal of an intermediate does not break any of the properties  (\Gnac)--(\consistency).}
	Indeed, $\supp(y - e) \cap \supp(y' - e) = \varnothing$ by construction, so $G^*$ satisfies (\Gnac) whenever $G$ does. Furthermore, it follows directly from the construction that no species in $G^*$ can take part in more than two reactions, that is, (\Hzero) also holds for $G^*$, as long as it already did for $G$.   
	
	It follows from \cite[Theorems 1 and 2]{mmf--feliu--wiuf-2015} that (\Nc)  {and (\consistency)} are preserved by the removal of an intermediate in the sense  { presented here}, which can be seen as a special case of the iterative removal of sets of intermediates or catalysts in the sense of \cite{mmf--feliu--wiuf-2015}. We omit the details;  see also Remark \ref{rem:collapse}.
\mybox	
\end{remark}

Given $Y_1, \ldots, Y_p \in \calS$, set $G_p = (\calS_p, \calC_p, \calR_p) := G$, and suppose that, for $j = p, \ldots, 1$, we recursively have $Y_1, \ldots, Y_j \in \calS_{j}$, that the species $Y_j$ is an intermediate of $G_{j}$, and then define $G_{j-1} = (\calS_{j-1}, \calC_{j-1}, \calR_{j-1})$ to be the reaction network obtained from $G_{j}$ by the removal of the intermediate $Y_j$.
	The reaction network $G_0$ obtained in the above construction is referred to as the {\em reduction of $G$ by the successive removal of intermediates $Y_p, \ldots, Y_1$.}

\begin{remark}\label{rem:ordering_irrelevance}
	The network $G_0$ does not depend on the order by which the intermediates $Y_p, \ldots, Y_1$ are removed.   {This follows from \cite[Theorem 3]{mmf--feliu--wiuf-2015} under assumption (G2) (see Remark \ref{rem:seven} also).  
	This observation further implies that
	 there is a unique minimal network obtained by iteratively removing intermediates (and cancelling catalysts) until this is no longer possible. Note, however, that $Y_j$ is an intermediate of $G_j$ and not necessarily of $G_p$; hence the set of intermediates to remove is not determined directly by $G$. }
	\mybox
\end{remark}

\begin{example}[The RKIP Network]\label{ex:RKIP}
	Consider the {\em RKIP network} discussed in \cite[Example 2]{angeli--deleenheer--sontag-2010}, displayed below in slightly modified notation as equations (\ref{eq:RKIP_1of3})--(\ref{eq:RKIP_3of3}).
		\begin{equation}\label{eq:RKIP_1of3}
			\begin{array}{c}
				R + K \longleftrightarrow RK
			\end{array}
		\end{equation}
		\begin{equation}\label{eq:RKIP_15of3}
			\begin{array}{c}
				RK + E_p \longleftrightarrow RKE_p \longrightarrow R + K_p + E
			\end{array}
		\end{equation}
		\begin{equation}\label{eq:RKIP_2of3}
			\begin{array}{c}
				M_p + E \longleftrightarrow M_pE \longrightarrow M_p + E_p
			\end{array}
		\end{equation}
		\begin{equation}\label{eq:RKIP_3of3}
			\begin{array}{c}
				K_p + P \longleftrightarrow K_pP \longrightarrow K + P
			\end{array}
		\end{equation}
 {This network fulfils hypotheses (G1)-(G4). } The reaction network obtained by the removal of the intermediate $M_pE$  { for which $e=M_p$} consists of (\ref{eq:RKIP_1of3}), (\ref{eq:RKIP_15of3}), (\ref{eq:RKIP_3of3}), plus the reaction $E \longrightarrow E_p$ ($M_p$ is cancelled out upon the removal of $M_pE$). We may further remove the intermediate $RKE_p$, then the intermediate $K_pP$ (leading to $e=P$ being also cancelled out), eventually obtaining
			
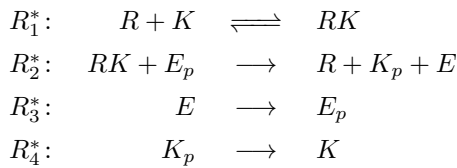
\[
				\begin{array}{rrcl}
					R_1^*\colon& R + K &\longleftrightarrow& RK \\[1ex]
					R_2^*\colon& RK + E_p &\longrightarrow& R + K_p + E \\[1ex]
					R_3^*\colon& E &\longrightarrow& E_p \\[1ex]
					R_4^*\colon& K_p &\longrightarrow& K
				\end{array}
			\]
		as the reduced reaction network.
\mydiamond
\end{example}

\begin{example}[Single-Phosphorylation Mechanism]\label{ex:one_site}
	Consider the one-site phosphorylation mechanism (\ref{eq:one_site}) discussed in the introduction. The reaction network obtained by the successive removal of intermediates $S_0E$ and $S_1F$ is given by
\[
	R_1^*\colon\quad S_0 \longrightarrow S_1
	\quad\quad\quad
	R_2^*\colon\quad S_1 \longrightarrow S_0\,.
\]
We emphasize that, in our formalism, the reduced network consists of the two reactions $R_1^*$ and $R_2^*$, and not of the single reversible reaction $S_0 \longleftrightarrow S_1$.
\mydiamond
\end{example}

\subsection{Invariance under the  {Successive Removal of Intermediates}}\label{subsec:main_result}

\begin{theorem}\label{thm:main1}
	Let $G$ be a reaction network satisfying {\em (\Gnac)-- { (\consistency)}}. Suppose $G^*$ is a reaction network obtained from $G$ by the successive removal of intermediates. Then $G^*$  {also satisfies} {\em (\Gnac)-- { (\consistency)}} and, furthermore,
	\begin{itemize}
		\item[{\em (}i\,{\em )}] the  R-graph of $G^*$ is   connected  { if and only if} the  R-graph of $G$ is  connected, and
		
		\item[{\em (}ii\,{\em )}] the R-graph of $G^*$ has the positive loop property  { if and only if} the R-graph of $G$ has the positive loop property.
	\end{itemize}
\end{theorem}

Furthermore, if these two graphical conditions are met, then (P1) and (P2) are invariant under the   {successive removal of} intermediates.
	
\begin{theorem}\label{thm:main2}
	Let $G$ be a reaction network satisfying {\em (\Gnac)-- {(\consistency)}}. Suppose $G^*$ is a reaction network obtained from $G$ by the successive removal of intermediates. Suppose, in addition, that the R-graph of $G^*$ has the positive loop property. 
		Let $N$ and $N^*$ be the stoichiometric matrices of $G$ and $G^*$, respectively, and $K$ and $K^*$ the  {orthants} constructed in Remark \ref{rem:orthant_cone} from the R-graphs of $G$ and $G^*$, respectively. Then,
	\[
		\ker N \cap K = \{0\}
		\quad
		\Leftrightarrow
		\quad
		\ker N^* \cap K^* = \{0\}\,,
	\]
	and
	\[
		\ker N \cap \interior K \neq \varnothing
		\quad
		\Leftrightarrow
		\quad
		\ker N^* \cap \interior K^* \neq \varnothing\,.
	\]
\end{theorem}

In view of Theorem \ref{thm:main1}, the graphical hypotheses on the R-graph in Proposition \ref{prop:ads-2010-thm2} for $G$ can be checked in $G^*$. And in view of Theorem \ref{thm:main2}, if these hypotheses are satisfied, then (P1) and (P2) can also be checked in $G^*$. The hypothesis of bounded-persistence in Proposition \ref{prop:ads-2010-thm2} can be checked using the graphical conditions in \cite{angeli--deleenheer--sontag-2007c}.  As shown in \cite{mmf--feliu--wiuf-2015}, these graphical conditions for bounded-persistence can also be checked in $G^*$ {---one need only to decouple the reversible reactions in order to apply the formalism in \cite{mmf--feliu--wiuf-2015} (see also \cite[pp. 612--615]{angeli--deleenheer--sontag-2010})}.  {Recall from Remark~\ref{rk:G4} that bounded-persistence together with (G3) imply (G4)}. We have thus devised a method to study the  {qualitative properties of generic and  global} convergence  {to equilibria } of a reaction network by analyzing a reduced network associated with it.  { In particular, we have the following corollary of Theorems \ref{thm:main1} and \ref{thm:main2}, and Proposition \ref{prop:ads-2010-thm2}.}

 {%
\begin{corollary}\label{cor:lets_refer_to_it}
Let $G$ be a bounded-persistent reaction network satisfying {\em (\Gnac)--(\consistency)} and {\em (\rone)--(\rthree)}, and let $G^*$ be a reaction network obtained from $G$ by the successive removal of intermediates. If the R-graph of $G^*$ is connected and  has the positive loop property, 
and $\ker N^* \cap \interior K^* \neq \varnothing$, then each stoichiometric compatibility class of $G$ has a unique equilibrium to which all positive solutions converge.
\end{corollary}}

The proofs of Theorems \ref{thm:main1} and \ref{thm:main2} will be given in Section \ref{sec:proofs}. We first illustrate the results with a few examples.  {We will see that the flexibility in treating a reversible reaction as two irreversible reactions is essential for Theorem \ref{thm:main2} to be true.}

\subsection{Examples}\label{subsec:examples}

\begin{example}[The RKIP Network]\label{ex:RKIP_revisited}
	Consider the RKIP network discussed in Example \ref{ex:RKIP}, which was reduced by the successive removal of intermediates $M_pE$, $RKE_p$ and $K_pP$. The R-graph of the reduced network is shown in Figure \ref{fig:r_of_reduced}. One can readily see that it has the positive loop property  {and is connected. 
		We conclude via Theorem \ref{thm:main1} that the R-graph of the original RKIP network has the same properties.}
	
	The R-graph of the reduced network yields, via Remark \ref{rem:orthant_cone}, the  {orthant} $K^* \! := \rplus^4$. Furthermore, each of the species $R, K, RK, E_p, E, K_p$ appears in exactly two reactions, once as a reactant, once as a product, both times with stoichiometric coefficient $1$. Therefore, $(1, 1, 1, 1) \in \kernel N^*$, showing that $\kernel N^* \cap \interior K^* \neq \varnothing$. It follows from Theorem \ref{thm:main2} that $\kernel N \cap \interior K \neq \varnothing$, where $N$ is the stoichiometric matrix of the original RKIP network, and $K$ is the  {orthant} obtained for its R-graph via Remark \ref{rem:orthant_cone}.
	
	The property of bounded-persistence for the flow of $G$ can also be checked directly on $G^*$ (see \cite[Theorems 1 and 2]{mmf--feliu--wiuf-2015} and Remark \ref{rem:collapse}).
	
	We conclude via  {Corollary \ref{cor:lets_refer_to_it}} that, under kinetic assumptions (\rone)--(\rthree), the RKIP network from Example \ref{ex:RKIP} has that property that each stoichiometric compatibility class has a unique equilibrium to which all trajectories starting with strictly positive concentrations converge.
\mydiamond	
\end{example}

\begin{figure}[!t]
\hfill
\begin{subfigure}[b]{0.45\textwidth}		
	\centering
	\begin{tikzpicture}[
		roundnode/.style={},
		squarednode/.style={rectangle, draw=black, minimum size=5mm},
		> = stealth, 
        shorten > = 1pt, 
        shorten < = 1pt, 
        auto,
   	    node distance = 1.5cm, 
       	semithick 
		]
		\node[squarednode]		(R1)								{\scriptsize $R_1^*$};
		\node[squarednode]		(R2)			[right=of R1]		{\scriptsize $R_2^*$};
		\node[squarednode]		(R3)			[below=of R2]		{\scriptsize $R_3^*$};
		\node[squarednode]		(R4)			[left=of R3]			{\scriptsize $R_4^*$};
 
		\draw[-]	 	(R1) edge node {\scriptsize $+$} (R2);
		\draw[-]	 	(R2) edge node {\scriptsize $+$} (R3);
		\draw[-]	 	(R2) edge node {\scriptsize $+$} (R4);
		\draw[-]	 	(R4) edge node {\scriptsize $+$} (R1);			
	\end{tikzpicture}

	\caption{R-graph of the reduced RKIP network}\label{fig:r_of_reduced}
\end{subfigure}
\hfill
\begin{subfigure}[b]{0.4\textwidth}		
	\centering
	\begin{tikzpicture}[
		roundnode/.style={},
		squarednode/.style={rectangle, draw=black, minimum size=5mm},
		> = stealth, 
        shorten > = 1pt, 
        shorten < = 1pt, 
        auto,
   	    node distance = 1.5cm, 
       	semithick 
		]
		\node[squarednode]		(R1)		at		(0, 0)		{\scriptsize $R^*_1$};
		\node[squarednode]		(R2)		at		(2, 0)		{\scriptsize $R^*_2$};
		 
		\draw[-]				(R1) edge node {\scriptsize $+$} (R2);			
	\end{tikzpicture}
	
		\vspace{0.5cm}
	\caption{R-graph of the reduced $n$-site phosphorylation network}\label{fig:r_of_phospho}
\end{subfigure}
\hfill

\caption{The  R-graphs of the reduced RKIP network from Example \ref{ex:RKIP} and of the reduced $n$-site phosphorylation network from Example \ref{ex:multi_site_sequential}. 
}\label{fig:of_reduced}
\end{figure}
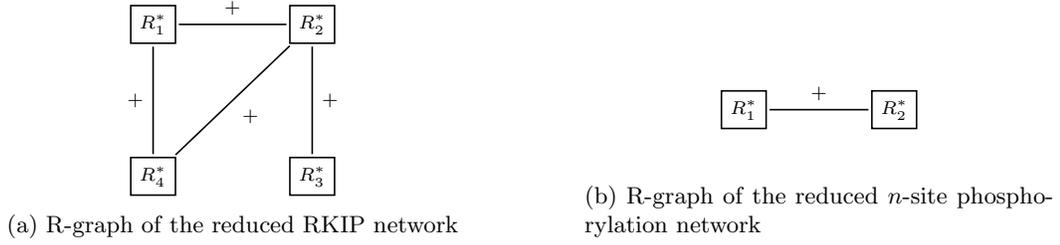

\begin{example}[Processive $n$-Site Phosphorylation Mechanism]\label{ex:multi_site_sequential}
	Consider the sequential and processive $n$-site phosphorylation mechanism described by the reaction network
	
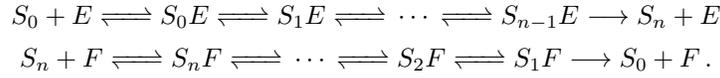
\[
		\begin{array}{c}
			S_0 + E \longleftrightarrow S_0E \longleftrightarrow S_1E \longleftrightarrow \cdots \longleftrightarrow S_{n-1}E \longrightarrow  { S_n + E} \\[1ex]
			S_n + F \longleftrightarrow S_nF \longleftrightarrow	 \cdots \longleftrightarrow S_2F \longleftrightarrow S_1F \longrightarrow S_0 + F\,.
		\end{array}
	\]
	(See \cite{conradi--shiu-2015} and references therein.) Note that the one-site mechanism from Example \ref{ex:one_site} is the special case when $n = 1$ of this mechanism. The reaction network obtained by the successive removal of the intermediates $S_0E, \ldots, S_{n-1}E, S_nF,$ $\ldots, S_1F$ is given by
\[
	R_1^*\colon\quad S_0 \longrightarrow S_1
	\quad\quad\quad
	R_2^*\colon\quad S_1 \longrightarrow S_0\,.
\]
The R-graph of the reduced network  has no loops, so it vacuously has the positive loop property   {(Figure \ref{fig:r_of_phospho})}.   {Furthermore, it is connected.
It follows from Theorem \ref{thm:main1} that the R-graph of the original $n$-site phosphorylation network has the same properties.}

The R-graph of the reduced network yields, via Remark \ref{rem:orthant_cone}, the  {orthant} $K^* := \rplus^2$. One can argue as in Example \ref{ex:RKIP_revisited} that $(1, 1) \in \kernel N^*$, showing that $\kernel N^* \cap \interior K^* \neq \varnothing$, and so $\kernel N \cap \interior K \neq \varnothing$ via Theorem \ref{thm:main2}. Finally, one can once again show that the flow of $G$ is bounded-persistent via \cite[Theorems 1 and 2]{mmf--feliu--wiuf-2015} and Remark \ref{rem:collapse}.
	
	It follows from  {Corollary \ref{cor:lets_refer_to_it}} that, under kinetic assumptions (\rone)--(\rthree), each stoichiometric compatibility class has a unique equilibrium to which all trajectories starting with strictly positive concentrations converge.
	
	 { This result is also proven in \cite{conradi--shiu-2015}. See also \cite{eithun--shiu-2017} for a class of reaction networks generalising this example  where the reduction approach presented here can be applied.   }
	
	 {In passing we make the observation that if $R^*_1$ and $R^*_2$ were treated as a single reversible reaction $S_0 \longleftrightarrow S_1$, then the corresponding orthant, say $\widetilde K$, and stoichiometric matrix, say $\widetilde N$, fulfils $\ker \widetilde N\cap \widetilde K=\{0\}\cap \rplus=\{0\}$. Hence, the kernel property of Theorem \ref{thm:main2} would not be invariant under reduction. Also note that in this case, we would only be able to infer the weaker convergence property implied by (P1) for the reduced network and not the stronger implied by (P2) as above. }
\mydiamond	
\end{example}

\begin{example}[A Phosphorelay]
Consider the general phosphorelay system studied in \cite{knudsen--feliu--wiuf-2012}. The underlying reaction network consists of the reactions
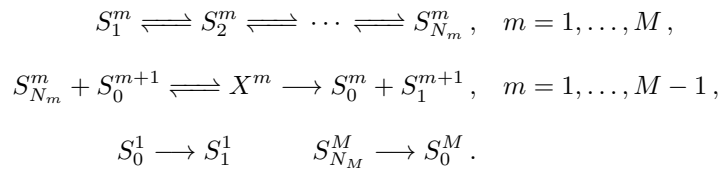
\begin{align*}
	S_1^m \longleftrightarrow S_2^m \longleftrightarrow \cdots \longleftrightarrow S_{N_m}^m\,,
	&\quad
	m = 1, \ldots, M\,,\\[8pt]
	S_{N_m}^m + S_0^{m+1} \longleftrightarrow X^m \longrightarrow S_0^m + S_1^{m+1}\,,
	& \quad
	m = 1, \ldots, M - 1\,, \\[8pt]
	S_0^1 \longrightarrow S_1^1
	\qquad
	\quad
	S_{N_M}^M \longrightarrow S_0^M\,. &
\end{align*}
For each $m \in [M]$ and each $n \in [N_m]$, $S_n^m$ represents the $m^{th}$ substrate (out of $M \geqslant 1$), phosphorylated at its $n^{th}$ site (out of $N_m \geqslant 1$), and $S_0^m$ corresponds to the unphosphorylated state of the $m^{th}$ substrate. The phosphate group can be transferred sequentially from site to site within the same substrate, or via the formation of an intermediate complex $X^m$ from $S_{N_m}^m$ to $S_1^{m+1}$, $m = 1, \ldots, M - 1$. The methods in \cite{angeli--deleenheer--sontag-2010} were employed in \cite{knudsen--feliu--wiuf-2012} to show that, under mass-action kinetics, the phosphorelay has a unique nonnegative equilibrium to which all solutions starting with positive concentrations converge.

First note that each species in the phosphorelay takes part in exactly two reactions. Thus, (\Hzero) is satisfied. Now
	\[
		\{S_1^1, S_2^1, \ldots, S_{N_1-1}^1, S_2^2, \ldots, S_{N_2-1}^2, \ldots, S_2^{M}, \ldots, S_{N_{M}-1}^{M}, S_{N_M}^M, X^1, \ldots, X^{M-1}\}
	\]
	is a set of intermediates. The network obtained after their removal is given by
	
	\hfill
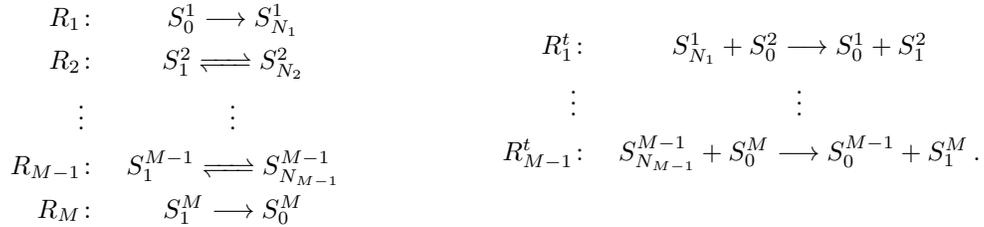
	\begin{minipage}[c]{0.35\textwidth}
		\[
			\begin{array}{rc}
				R_1\colon & S_0^1 \longrightarrow S_{N_1}^1 \\[1ex]
				R_2\colon & S_1^2 \longleftrightarrow S_{N_2}^2 \\[1ex]
				 { \vdots \phantom{x}} & \vdots \\[1ex]
				R_{M-1}\colon & S_1^{M-1} \longleftrightarrow S_{N_{M-1}}^{M-1} \\[1ex]
				R_M\colon & S_1^M \longrightarrow S_0^M
			\end{array}
		\]
	\end{minipage}%
	\hfill
	\begin{minipage}[c]{0.6\textwidth}
		\[
			\begin{array}{rc}
				R_1^t\colon & S_{N_1}^1 + S_0^2 \longrightarrow S_0^1 + S_1^2 \\[1ex]
				 { \vdots \phantom{x}} & \vdots \\[1ex]
				R_{M-1}^t\colon & S_{N_{M-1}}^{M-1} + S_0^M \longrightarrow S_0^{M-1} + S_1^M\,.
			\end{array}
		\]
	\end{minipage}%
	\hfill
	
	\vskip 1ex
	
	 {The  R-graph of the reduction is sketched in Figure  \ref{fig:phosphorelay_R}, where it  can be readily seen that it is connected and has the positive loop property}. It follows from Theorem \ref{thm:main1} that the R-graph of the original network  {also fulfils these properties}. 
	
	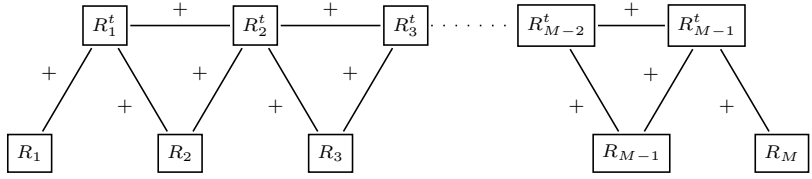
\begin{figure}[!t]
	\centering
	
	\centering
	\begin{tikzpicture}[
		roundnode/.style={circle, draw=green!60, fill=green!5, very thick, minimum size=5mm},
		squarednode/.style={rectangle, draw=black, minimum size=5mm},
		> = stealth, 
        shorten > = 1pt, 
        shorten < = 1pt, 
        auto,
       	semithick 
		]
		\node[squarednode]		(R1)		at		(0, 0)		{\scriptsize $R_1$};
		\node[squarednode]		(Rt1)		at		(1, 1.7)		{\scriptsize $R_1^t$};
		\node[squarednode]		(R2)		at		(2, 0)		{\scriptsize $R_2$};
		\node[squarednode]		(Rt2)		at		(3, 1.7)		{\scriptsize $R_2^t$};
		\node[squarednode]		(R3)		at		(4, 0)		{\scriptsize $R_3$};
		\node[squarednode]		(Rt3)		at		(5, 1.7)		{\scriptsize $R_3^t$};
		\node[squarednode]		(RMm1)		at		(8, 0)		{\scriptsize $R_{M-1}$};
		\node[squarednode]		(RtMm2)		at		(7, 1.7)		{\scriptsize $R_{M-2}^t$};
		\node[squarednode]		(RtMm1)		at		(9, 1.7)		{\scriptsize $R_{M-1}^t$};
		\node[squarednode]		(RM)		at		(10, 0)		{\scriptsize $R_{M}$};
 
		\draw[-]					(R1) edge node {\scriptsize $+$} (Rt1);
		\draw[-]					(R2) edge node {\scriptsize $+$} (Rt1);
		\draw[-]					(Rt1) edge node {\scriptsize $+$} (Rt2);
		\draw[-]					(R2) edge node {\scriptsize $+$} (Rt2);
		\draw[-]					(R3) edge node {\scriptsize $+$} (Rt2);
		\draw[-]					(Rt2) edge node {\scriptsize $+$} (Rt3);
		\draw[-]					(R3) edge node {\scriptsize $+$} (Rt3);
		\draw[loosely dotted]		(Rt3) -- (RtMm2);
		\draw[-]					(RMm1) edge node {\scriptsize $+$} (RtMm2);
		\draw[-]					(RMm1) edge node {\scriptsize $+$} (RtMm1);
		\draw[-]					(RM) edge node {\scriptsize $+$} (RtMm1);
		\draw[-]					(RtMm2) edge node {\scriptsize $+$} (RtMm1);
	\end{tikzpicture}
	
	\caption{The  R-graph of the reduction of the phosphorelay system.}\label{fig:phosphorelay_R}
	\end{figure}
	
	The R-graph of the reduced network (Figure \ref{fig:phosphorelay_R}) yields, via Remark \ref{rem:orthant_cone}, the  {orthant} $K^* := \rplus^{2M-1}$. The same argument as in the previous two examples shows that $(1, \ldots, 1) \in \r^{2M-1}$ belongs to the kernel of the stoichiometric matrix $N^*$ of the reduced network. Thus, $\ker N^* \cap \interior K^* \neq \varnothing$, and so $\ker N \cap \interior K \neq \varnothing$ by Theorem \ref{thm:main2}. As before, bounded-persistent follows via \cite[Theorems 1 and 2]{mmf--feliu--wiuf-2015} and Remark \ref{rem:collapse}.
		
	We conclude via  {Corollary \ref{cor:lets_refer_to_it}} that, within each stoichiometric compatibility class, there exists a unique nonnegative equilibrium to which all solutions starting with strictly positive concentrations converge.
\mydiamond	
\end{example}

\subsection{Further Comments on  { the Definition of Intermediates}}\label{subsec:exceptions}

Our working definition of intermediates in this paper is  {not quite as general as } 
in \cite{mmf--feliu--wiuf-2015}, where a related study of invariance of qualitative properties of reaction networks under the removal of sets of intermediates and catalysts was carried out. We conclude this section with a discussion of the differences, 
 {and give an example of what goes wrong with the kinds of intermediates precluded in our working definition.}

In \cite{mmf--feliu--wiuf-2015}, a species $Y$ would still be considered an intermediate if $y = y'$ in (I2).  In this case, $G^*$ is defined by simply removing the reaction $y \longleftrightarrow Y$ from $G$ (or removing the reactions $y \longrightarrow Y$ and $Y \longrightarrow y$, if that is the case). 
 { The positive loop property of the R-graph is not invariant under the removal of intermediates of this type. To see this, consider the reaction network $G$
\[
  R_1\colon  \quad A + B \longleftrightarrow Y \quad\quad\quad R_2\colon \quad A \longleftrightarrow B\,.
\]
The reaction network obtained by removing $Y$ as described above is
\[
	G^*\colon \quad\quad A \longleftrightarrow B\,.
\]
The R-graph of $G^*$  has the positive loop property, since it has only one vertex, while that of $G$ has not. Indeed, the R-graph of $G$ is 
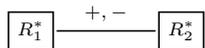
\begin{center}
\begin{tikzpicture}[
		roundnode/.style={},
		squarednode/.style={rectangle, draw=black, minimum size=5mm},
		> = stealth, 
        shorten > = 1pt, 
        shorten < = 1pt, 
        auto,
   	    node distance = 1.5cm, 
       	semithick 
		]
		\node[squarednode]		(R1)		at		(0, 0)		{\scriptsize $R^*_1$};
		\node[squarednode]		(R2)		at		(2, 0)		{\scriptsize $R^*_2$};
		 
		\draw[-]				(R1) edge node {\scriptsize $+,-$} (R2);			
	\end{tikzpicture}
	\end{center}
which clearly has a loop with an odd number of negative edges.}
 
 {The same happens if we consider the case of irreversible reactions $A+B\longrightarrow Y$ and $Y \longrightarrow A+B$, rather than the single reversible reaction $A + B \longleftrightarrow Y$.}

 {Also it might be that the R-graph of $G^*$ is connected, while that of $G$ is not. For example, consider the reaction network $G$
\[
  R_1\colon  \quad A + B \longleftrightarrow Y \quad\quad\quad R_2\colon \quad A \longleftrightarrow C\quad\quad\quad R_3\colon \quad B \longleftrightarrow D\,.
\]
The reaction network obtained by removing $Y$  is
\[
	G^*\colon \quad\quad  A \longleftrightarrow C\quad\quad\quad R_3\colon \quad B \longleftrightarrow D.
\]
The R-graph for $G^*$ has two connected components, while that of $G$ is connected.
}

\section{Proofs of Theorems \ref{thm:dsrgraph}, \ref{thm:main1} and \ref{thm:main2}}\label{sec:proofs}

\subsection{Proof of Theorem \ref{thm:dsrgraph}}
 {In order to prove Theorem~\ref{thm:dsrgraph}, we start with an auxiliary lemma. }
 {
Given a vector subspace $V\subseteq \r^n$, we say that a nonnegative vector $\omega \in V\cap \r^n_{\geqslant 0}$ has \emph{minimal support} if there does not exist a  nonzero vector $\omega'\in V\cap \r^n_{\geqslant 0}$ with support strictly included in $\supp \omega$.

\begin{lemma}\label{lem:great}
Assume $G$ fulfils {\em (G1), (G2), (G4)}, the set $\calR_{\rightarrow}$ is nonempty,  and the SR-graph of $G$ is connected. Then,
\begin{itemize}
\item[(i)] If   $v \in \ker N$ fulfils $v_j>0$ for all $j\in [m]$ for which $R_j$ is irreversible, then $v_j\neq 0$ for all $j \in [m]$ for which $R_j$ is reversible.
\item[(ii)] Every species  takes part in exactly two reactions in $\mathcal{R}$.
\item[(iii)] Every simple loop in $G_{SR}$  defines a vector $\omega\in \Gamma^{\perp}$ with support the set of species that are vertices in the loop.
\item[(iv)] For every vector $\omega\in \Gamma^{\perp}\cap \r^n_{\geqslant 0}$ with minimal support, there is a  directed simple  loop in $G_{SR}^\rightarrow$ with species vertex all the species in $\supp(\omega)$.
\end{itemize}
\end{lemma} 

\begin{proof} 
Recall that the entries of the stoichiometric matrix $N$ are denoted by $N_{ij}$. In this proof we use repeatedly that, in view of (G1), $N_{ij}\neq 0$ whenever $S_i$ is part of   reaction $R_j$. 

(i) Let $R_j\in \mathcal{R}$ and choose any $R_{j_0}\in \mathcal{R}_{\rightarrow}\neq \varnothing$. Then $v_{j_0}> 0$ by assumption.
 By hypothesis, there exists 
 a simple (undirected) path from $R_{j_0}$ to $R_j$:
$$R_{j_0}  \   \text{---} \   S_{i_1}   \   \text{---} \   R_{j_1}  \   \text{---} \    \dots  \   \text{---} \   R_{j_{k-1}}   \   \text{---} \    S_{i_k}  \   \text{---} \  R_{j_k}=R_j.   $$
By (G2), this path tells us that the $i_\ell$-th row of $N$ has exactly two entries different from zero,   $N_{i_\ell j_{\ell-1}}$ and $N_{i_\ell j_{\ell}}$, for all $\ell \in [k]$.
Since the scalar product of each of these rows  with $v$ is zero, we have
\begin{equation}\label{eq:n} N_{i_\ell j_{\ell-1}}v_{j_{\ell-1}} +  N_{i_\ell j_{\ell}} v_{j_\ell}  =0,\qquad \ell \in [k].
\end{equation}
Thus
\begin{equation}\label{eq:n2} 
v_{j_\ell} = \frac{-N_{i_\ell j_{\ell-1}}v_{j_{\ell-1}}}{N_{i_\ell j_\ell}},\ \ell \in [k], \quad \Rightarrow \quad   v_j=v_{j_k} = (-1)^k \left(\prod_{\ell=1}^{k} \frac{N_{i_\ell j_{\ell-1}}}{N_{i_\ell j_\ell}} \right) v_0 \neq 0.
\end{equation}

\medskip
In the rest of the proof we choose a vector   $v\in \ker N$  as in (i), which exists by assumption (G4).

\medskip
(ii) If $S_i$ is part of only one reaction $R_j$, then the $i$-th row of $N$ has only one nonzero entry, $N_{ij}$. The equality $Nv=0$ implies $v_j=0$, contradicting (i).

\medskip
(iii) Write the loop as
\begin{equation}\label{eq:loop}
S_{i_1} \   \text{---} \   R_{j_1}   \ \text{---} \  S_{i_2} \  \text{---} \  \dots  \  \text{---}      \  S_{i_k}\  \text{---}   \  R_{j_k}  \ \text{---} \  S_{i_1}.   
\end{equation}
Consider the submatrix $\widetilde{N}$ of $N$ given by the rows $i_1,\dots,i_k$. 
By (G2), only the columns $j_1,\dots,j_k$ of $\widetilde{N}$ are nonzero. Since $\widetilde{N} v =0$ and all entries of $v$ are different from zero, it follows that   the rank of the matrix is (at most) $k-1$.
Thus there exists a nonzero vector $\widetilde{\omega}\in \r^{k}$ such that $\widetilde{\omega}^t\widetilde{N} =0$.
Since each column and row of $\widetilde{N}$ has exactly two nonzero entries, all components of $\widetilde{\omega}$ are nonzero.
 The vector $\omega\in \r^n$ defined by $\omega_\ell=\widetilde{\omega}_\ell$ for $\ell\in \{i_1,\dots,i_k\}$, and zero otherwise, satisfies $\omega\in \Gamma^\perp$ and   it has support  $\{S_{i_1},\dots,S_{i_k}\}$.
 
\medskip
(iv)
Let $S_{i_1}\in \supp \omega$ and $R_{j_1}$ be one of the two reactions $S_{i_1}$ is involved in. 
Since $\omega^t N=0$ and the components of $\omega$ are nonnegative, $R_{j_1}$ involves at least one other species  $S_{i_2}\in \supp \omega$  for which $N_{i_2j_1}$ and $N_{i_1j_1}$ have opposite nonzero sign.  
Let $R_{j_2}$ be the other reaction involving $S_{i_2}$ by (ii). 
 Again $R_{j_2}$ must involve another species $S_{i_3}\in \supp \omega$, such that $N_{i_2j_2}$ and $N_{i_3j_2}$ have opposite nonzero sign. Proceeding in this way we create a  path   in  $G_{SR}$:
$$  S_{i_1} \   \text{---} \   R_{j_1}   \ \text{---} \  S_{i_2} \  \text{---} \  R_{j_2}  \ \text{---} \  S_{i_3} \  \text{---} \   \dots  $$
where all species vertices are in the support of $\omega$.
At some point the path meets a species or a reaction   already considered, creating a simple loop. By (iii), this loop generates a vector $\omega'\in \Gamma^{\perp}$ with support the set of species in the loop. We will show that the loop involves all the species in  $\supp \omega$. 
Write the loop as in \eqref{eq:loop} and consider the matrix $\widetilde{N}$ as  in the proof of (iii). The two nonzero entries of the columns $j_1,\dots,j_k$ of $\widetilde{N}$ have opposite sign by construction of the loop. It follows that $\omega'_{j_1}, \dots,\omega_{j_k}'$ have the same sign.  Thus $\omega'$ can be chosen with nonnegative entries. 
Since $\omega$ has minimal support, $\supp \omega=\supp \omega'=\{S_{i_1},\dots,S_{i_k}\}$.

\medskip
It remains  to show that  this loop is directed. 
Equation \eqref{eq:n} holds for $\ell \in [k]$ by defining $j_0=j_k$. 
 {It implies that $N_{i_\ell j_{\ell-1}}v_{j_{\ell-1}}$ and $N_{i_\ell j_{\ell}} v_{j_\ell} $ have opposite signs and are nonzero for all $\ell\in [k]$. 
Assume $N_{i_1j_{k}} v_{j_{k}}>0$. Then $N_{i_1j_{1}} v_{j_{1}}<0$ and thus $N_{i_2j_{1}} v_{j_{1}}>0$. We iterate this argument  along the loop for $\ell \in [k]$ to conclude that $N_{i_\ell j_{\ell-1}} v_{j_{\ell-1}}  >0$
 and $N_{i_\ell j_\ell} v_{j_\ell} <0$ for all $\ell\in [k]$.}
 
 { Recall also that $N_{i_\ell j_{\ell}}$ and $N_{i_{\ell +1}j_{\ell}}$ have opposite nonzero signs (with $i_{k+1}=i_1$).}
By Remark~\ref{rem:SR_labeling}, it is enough to show that for every irreversible reaction $R_{j_\ell}$ that is a vertex in the loop, then $S_{i_{\ell}}$ is not part of its product. If this were the case, then we would have $N_{i_{\ell} j_\ell}>0$, and since the reaction is irreversible, $v_{j_\ell}>0$, contradicting that $N_{i_\ell j_\ell} v_{j_\ell} <0$. Therefore  the loop is directed.

 If $N_{i_1j_{k}} v_{j_{k}}<0$, then we conclude that the loop is directed in the reverse direction.
 \mybox
 \end{proof}

We are now ready to prove Theorem~\ref{thm:dsrgraph}. 
 {First observe that since every species is part of at least one reaction, the SR-graph is connected if and only if the R-graph is connected by definition. Further, if the directed SR-graph is R-strongly connected, then the R-graph is connected by construction. It remains  to show that if the SR-graph is connected, then the directed SR-graph is R-strongly connected.} 

 {For this, assume that the SR-graph of G is connected. }
 If all reactions of $G$ are reversible, then edges in both direction exist between any pair of vertices of $G_{SR}^{\rightarrow}$. Thus $G_{SR}^{\rightarrow}$ is R-strongly connected.
	
Assume now that there exists at least one irreversible reaction. 
Since by hypothesis the SR-graph of $G$ is connected, it is enough to show that every edge $(S_{i_\ell},R_{j_\ell})$ in $G_{SR}^{\rightarrow}$ is part of a directed simple loop in $G_{SR}^{\rightarrow}$. Indeed, given a pair of reactions $R$ and $R'$,  there exists an undirected path in the 
SR-graph of $G$:
$$ R= R_{j_1}   \   \text{---} \  S_{i_2} \   \text{---} \   R_{j_2}    \  \text{---} \  \dots  \  \text{---}      \  S_{i_{k}}\  \text{---}   \  R_{j_k}  =R' $$
An edge from a reaction vertex to a species vertex always exists provided the two vertices are connected in $G_{SR}$. If only the edge $(R_{j_\ell},S_{i_\ell})$ exists in $G_{SR}^{\rightarrow}$ for one of the edges in the path in $G_{SR}$, then we can use the directed simple loop to replace it with a directed path from $S_{i_\ell}$ to $R_{j_\ell}$. In this way we construct a directed path from $R$ to $R'$.

We use the following result. There exist 
\begin{equation}\label{eq:extremerays}
\omega^1,\dots,\omega^d \in \Gamma^{\perp} \cap \r^n_{\geqslant 0}
\end{equation}
 such that each 
 $\omega^i$, $i\in [d]$, has minimal support (in the sense that there is not another $\omega' \in \Gamma^{\perp} \cap \r^n_{\geqslant 0}$ with support contained in $\omega^i$) and any other $\omega \in \Gamma^{\perp} \cap \r^n_{\geqslant 0}$ can be written as a linear combination of $\omega^1,\dots,\omega^d$ with nonnegative coefficients. (These are known as the \emph{extreme rays} of the polyhedral cone $\Gamma^{\perp} \cap \r^n_{\geqslant 0}$ \cite{rockafellar}.)

Consider an edge $(R_{j_\ell},S_{i_\ell})$ in the directed SR-graph of $G$ for $R_j$ irreversible. Then $S_{i_\ell}$ is part of the product of $R_{j_\ell}$ and the edge $(S_{i_\ell},R_{j_\ell})$ is not in $G_{SR}^{\rightarrow}$. 
Let $R_{j_{\ell-1}}$ be the other reaction involving $S_{i_\ell}$, Lemma~\ref{lem:great}(ii). 
Since the network is conservative, at least one of the vectors in \eqref{eq:extremerays}, say $\omega=\omega^i$ satisfies $\omega_{i_\ell}>0$.
By Lemma~\ref{lem:great}(iv), there is a directed simple loop involving $S_{i_\ell}$. 
Such a loop involves the two reactions $R_{j_{\ell-1}}$ and $R_{j_{\ell}}$ and since $R_{j_{\ell}}$ is irreversible, it is necessarily of the form
$$  R_{j_{\ell}}  \longrightarrow    S_{i_\ell}    \longrightarrow R_{j_{\ell-1}}    \longrightarrow \dots \longrightarrow  R_{j_{\ell}}.$$
This gives a directed simple loop having $ R_{j_{\ell}} $ as a vertex, finishing the proof.

\mybox
	
	}

\subsection{Proof of Theorem \ref{thm:main1}({\em i\,})}

Note that it suffices to prove Theorems \ref{thm:main1} and \ref{thm:main2} for the removal of a single intermediate $Y$. The general result then follows by induction on the number of intermediates successively removed. We have four cases to consider, depending on how $Y$ appears in $G$, all of which are captured by
\begin{equation}\label{eq:cases_i_and_ii}
	 { y= } \widehat y + \sum_{i = 1}^p \gamma_iE_i \ \text{---} \ Y \ \text{---} \ \widehat y' + \sum_{i = 1}^p \gamma_iE_i = 	 { y'}\,,
\end{equation}
where each `---' may mean either `$\longrightarrow$' or `$\longleftrightarrow$,' 
\[
	e := \sum_{i = 1}^p \gamma_iE_i
\]
may be an empty sum, $\supp \widehat y \cap \supp \widehat y' = \varnothing$,  { and $\supp \widehat y\neq \varnothing$ and $\supp \widehat y'\neq \varnothing$.}

	Denote 
	\[
		R_Y\colon y \ \text{---} \ Y\,,
		\quad\quad\quad
		R_Y'\colon Y \ \text{---} \ y'\,.
	\]

%
%
 {In view of Theorem~\ref{thm:dsrgraph}, we show that the SR-graph of $G^*$ is connected if and only  the SR-graph of $G$ is connected. }
By reordering the species in $\calS = \{S_1, \ldots, S_n\}$, if necessary, we may write $\widehat y = \alpha_1S_1 + \cdots + \alpha_kS_k$ and $\widehat y' = \alpha'_1S'_1 + \cdots + \alpha'_{k'}S'_{k'}$ for some $\alpha_1, \ldots, \alpha_k>0$, $\alpha'_1, \ldots, \alpha'_{k'} > 0$ and some $S_1, \ldots, S_k, S'_1, \ldots, S'_{k'} \in \calS$. Figure \ref{fig:directed_sr_proof_b} illustrates the  SR-graph of $G$. Note that there is always a  path $R_Y\   \text{---} \   Y\   \text{---} \  R_Y'$ in the graph. This path is replaced by the reaction vertex $R^* = \widehat{y} \ \text{---} \  \widehat{y'}$ in the  SR-graph of $G^*$ (see Figure \ref{fig:directed_sr_proof_c}).

\medskip
 ($\Rightarrow$) Suppose the SR-graph of $G^*$ is  connected. 
For any $R \in \calR  \backslash \{R_Y, R'_Y\}$, there exists a  simple path
\[
	R \   \text{---}^* \    \cdots \   \text{---}^* \  S \   \text{---}^* \   R^* 
\]
in the  SR-graph of $G^*$ connecting $R$ and $R^*$ with $S\in \mathcal{S}$. Here, we annotate the edges with a $*$ to emphasize that the path is in the SR-graph of $G^*$ rather than $G_{SR}$. None of the reaction vertices of the path are $R_Y$ or $R_Y'$ and the species vertices are different from $Y$ and the species in $\supp e$.
 
If $S$ is part of the reactant of $R^*$, then it is part of the reactant of $R_Y$ and we have a path from $R$ to $R_Y$. Similarly, if $S$ is part of the product of $R^*$, then  we have a path from $R$ to $R_Y'$. Combining these paths with the path $R_Y  \   \text{---} \ Y  \   \text{---} \ R_Y'$ if necessary, we obtain a path connecting any reaction $R$ to $R_Y'$, showing that $G_{SR}$ is also connected.

\medskip
($\Leftarrow$) Now suppose $G_{SR}$ is  connected. 
We   show that the SR-graph of $G^*$ is also connected by showing that  for any $R \in \calR^* \backslash \{R^*\}$, there is a path  connecting $R$ to $R^*$. 	
Since $R \neq R^*$, we have $R \in \calR$. By  connectedness, there exists a simple path
\[
	R'_Y   \   \text{---} \ S_{i_1}   \   \text{---} \  R_{j_1}   \   \text{---} \ S_{i_2}  \   \text{---} \ \cdots   \   \text{---} \ S_{i_k}  \   \text{---} \ R
\]
connecting $R'_Y$ and $R$ in  $G_{SR}$. 
Note that only one of the species $Y$ and the species in $\supp e$ can be part of the path, as $R'_Y$ would otherwise be part of this path twice. If this is the case, then the species must be $S_{i_1}$  and $R_{j_1}=R_Y$.

If the path does not go through $R_Y$, 
then all reactions and species vertices belong to $G^*$ and thus the path gives a path connecting $R^*$ and $R$ by merely replacing $R'_Y$ by $R^*$. 

If the path does go through $R_Y$, in other words, if $R_{j_1}=R_Y$, then the path 
\[
	R^*   \   \text{---}^* \ S_{i_2}  \   \text{---}^* \ \cdots   \   \text{---}^* \  S_{i_k}   \   \text{---}^* \ R
\]
is a path connecting $R^*$ and $R$. This finishes the proof. \mybox

\begin{figure}[!tb]
\hfill
\begin{subfigure}{0.45\textwidth}
	\centering
	\begin{tikzpicture}[
		roundnode/.style={},
		squarednode/.style={rectangle, draw=black, thick, minimum size=5mm},
		> = stealth, 
        shorten > = 2pt, 
        shorten < = 2pt, 
        auto,
       	semithick 
		]
		\node[squarednode]		(RY)		at		(1, 0)			{\scriptsize $R_Y$};
		\node[roundnode]			(Y)			at		(2.5, 1)			{\scriptsize $Y$};
		\node[squarednode]		(RYprime)	at		(4, 0)			{\scriptsize $R'_Y$};
		\node[roundnode]			(Sk)		at		(0, -1)			{\scriptsize $S_k$};
		\node					(luvdots)	at		(-0.2, 0)			{$\vdots$};
		\node[roundnode]			(S1)		at		(0, 1)			{\scriptsize $S_1$};
		\node[roundnode]			(E1)		at		(2.5, -0.5)		{\scriptsize $E_1$};
		\node					(Edots)		at		(2.5, -1)		{$\vdots$};
		\node[roundnode]			(Ep)		at		(2.5, -2)		{\scriptsize $E_p$};
		\node[roundnode]			(S1prime)	at		(5, 1)			{\scriptsize $S'_{1}$};
		\node					(ruvdots)	at		(5.2, 0)			{$\vdots$};
		\node[roundnode]			(Skprime)	at		(5, -1)			{\scriptsize $S'_{k'}$};
		
		\draw[-] (S1) edge node {$+$} (RY);
		\draw[-] (Sk) edge node {$+$} (RY);
		\draw[-] (RY) edge node {$+$} (E1);
		\draw[-] (RY) edge node {$+$} (Ep);
		\draw[-] (RY) edge node {$-$} (Y);
		\draw[-] (Y) edge node {$+$} (RYprime);
		\draw[-] (RYprime) edge node {$-$} (S1prime);
		\draw[-] (RYprime) edge node {$-$} (Skprime);
		\draw[-] (E1) edge node {$-$} (RYprime);
		\draw[-] (Ep) edge node {$-$} (RYprime);
	\end{tikzpicture}
	\caption{SR-graph of $G$}\label{fig:directed_sr_proof_b}
\end{subfigure}%
\hfill
\begin{subfigure}{0.45\textwidth}
	\centering
	\begin{tikzpicture}[
		roundnode/.style={},
		squarednode/.style={rectangle, draw=black, thick, minimum size=5mm},
		> = stealth, 
        shorten > = 2pt, 
        shorten < = 2pt, 
        auto,
       	semithick 
		]
		\node					(lvdots)			at		(-0.2, 0)		{$\vdots$};
		\node[roundnode]			(S1)			at		(0, 1.5)			{$S_1$};
		\node[roundnode]			(Sk)			at		(0, -1.5)			{$S_k$};
		\node[squarednode]		(Rstar)			at		(1.5, 0)			{$R^*$};
		\node					(rvdots)			at		(3.2, 0)			{$\vdots$};
		\node[roundnode]			(S1prime)		at		(3, 1.5)			{$S'_1$};
		\node[roundnode]			(Skprime)		at		(3, -1.5)			{$S'_{k'}$};
		
		\draw[-] (S1) edge node {$+$} (Rstar);
		\draw[-] (Sk) edge node {$+$} (Rstar);
		\draw[-] (Rstar) edge node {$-$} (S1prime);
		\draw[-] (Rstar) edge node {$-$} (Skprime);
	\end{tikzpicture}
	\caption{SR-graph of $G^*$}\label{fig:directed_sr_proof_c}
\end{subfigure}%
\hfill
\caption{ {Local structure of the }SR-graphs of $G$ and $G^*$. }\label{fig:proof}
\end{figure}
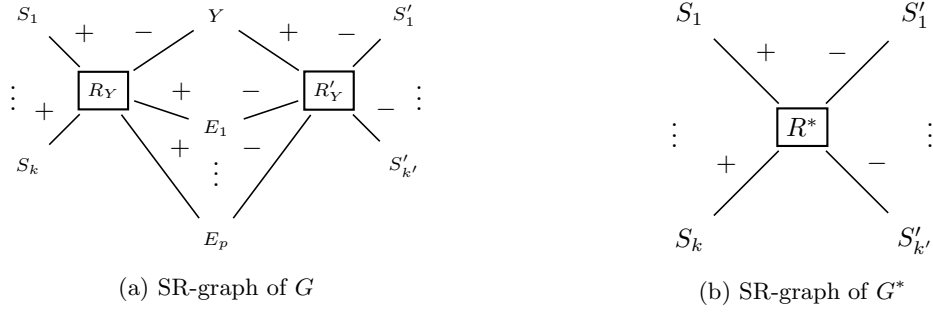

\subsection{Proof of Theorem \ref{thm:main1}({\em ii}\,)} \label{subsec:th(ii)}

To study the positive loop property of the R-graph of $G$ and $G^*$,  we will use Proposition \ref{prop:posloop}, together with the lemma below.

\begin{lemma}\label{lemm:eloops}
	A simple loop in the SR-graph is an e-loop  {if and only if} it contains an even number of segments $S_i \ \text{---} \ R_j \ \text{---} \ S_k$,  { where $S_i, S_k \in \calS$ and $R_j \in \calR$}, such that $L_{SR}(\{S_i, R_j\}) = L_{SR}(\{R_j, S_k\})$. 
\end{lemma}

\begin{proof}
	See \cite[Lemma 4.4]{angeli--deleenheer--sontag-2010}.
	\mybox
\end{proof}

Let $R_Y, R'_Y, R^*, \alpha_1, \ldots, \alpha_k, \alpha'_1, \ldots, \alpha'_{k'}, S_1, \ldots, S_k, S'_1, \ldots, S'_{k'}$ be as in the proof of Theorem \ref{thm:main1}({\em i}\,). 

On the one hand, $L$ is a simple e-loop (respectively, o-loop) of $G^*_{SR}$ which does not go through $R^*$  { if and only if it} is a simple e-loop (respectively, o-loop) of $G_{SR}$ which does not go through either of $R_Y$, $Y$ and $R_Y'$. 

On the other hand, the simple loops of $G^*_{SR}$ which go through $R^*$ are in one-to-one correspondence with the simple loops of $G_{SR}$ going through $R_Y$, $Y$ or $R_Y'$. This correspondence is established as follows. Any simple loop $L$ in $G^*_{SR}$ which goes through $R^*$ has the form
\[
	R^* \ \text{---} \ S_{out} \ \text{---} \ R_{i_1} \ \text{---} \ S_{i_1} \ \text{---} \ \cdots \ \text{---} \ S_{i_\ell} \ \text{---} \ R_{i_\ell} \ \text{---} \ S_{in} \ \text{---} \ R^*
\]
for some pairwise distinct $R_{i_1}, \ldots, R_{i_\ell} \in \forcalR^* \cap \forcalR$ and $S_{out}, S_{i_1}, \ldots, S_{i_\ell}, S_{in} \in \calS^* \cap  {\calS}$  {and not in $\supp e$}. If $S_{out}, S_{in} \in \{S_1, \ldots, S_k\}$ (respectively, $S_{out}, S_{in} \in \{S'_1, \ldots, S'_{k'}\}$), then we need only replace $R^*$ with $R_Y$ (respectively, $R_Y'$). If $S_{out}$ belongs to one of the sets $\{S_1, \ldots, S_k\}$ and $\{S'_1, \ldots, S'_{k'}\}$, and $S_{in}$ belongs to the other, then we need only replace $R^*$ by the segment $R_Y \ \text{---} \ Y \ \text{---} \ R'_Y$. Now note that this correspondence also takes e-loops (respectively, o-loops) to e-loops (respectively, o-loops). Indeed, this follows from Lemma \ref{lemm:eloops}. If $S_{in}$ and $S_{out}$ belong to the same set, the simple loop in $G_{SR}$ has the same number of edges as its corresponding loop in $G^*_{SR}$, and they both have the same sign pattern. If $S_{in}$ and $S_{out}$ belong to different sets, then the consecutive edges in the segments
\[
S_{in} \ \text{---} \ R^* \ \text{---} \ S_{out}
\quad
\text{and}
\quad
S_{in} \ \text{---} \ R_Y \ \text{---} \ Y \ \text{---} \ R_Y' \ \text{---} \ S_{out}
\]
have opposite signs, as one can see in Figures \ref{fig:directed_sr_proof_b} and \ref{fig:directed_sr_proof_c}, so the number of segments $S_i \ \text{---} \ R_j \ \text{---} \ S_k$ such that $S_i \ \text{---} \ R_j$ and $R_j \ \text{---} \ S_k$ have the same sign does not change from the path in $G^*_{SR}$ to the corresponding path in $G_{SR}$.

We conclude that every simple loop in $G_{SR}$ is an e-loop  { if and only if} every simple loop in $G^*_{SR}$ is an e-loop. Since (\Hzero) holds, it follows from Proposition \ref{prop:posloop} that the R-graph of $G$ has the positive loop property  { if and only if} the R-graph of $G^*$ also has the positive loop property. \mybox

\subsection{Proof of Theorem \ref{thm:main2}}

Once we understand the relationships between $\kernel N$ and $\kernel N^*$ and between $K$ and $K^*$, the proof of the theorem will follow somewhat effortlessly.

\subsubsection*{Relationship between $\kernel N$ and $\kernel N^*$}

We first consider the case in which $e$ is nontrivial. By reordering the species and reactions so that $Y, E_1, \ldots, E_p$, and the reactions $y \ \text{---} \ Y$ and $Y \ \text{---} \ y'$ appear at the end, if necessary, we may write the stoichiometric matrices $N$ and $N^*$ of, respectively, $G$ and $G^*$ as
\[
	N
	=
	\left[
		\begin{array}{ccc|c|c}
			 &  &  & - \alpha_1 & \alpha'_1 \\
			 & N_c^* &  & \vdots & \vdots \\
			 &  &  & - \alpha_n & \alpha'_n \\ \hline
			 0 & \cdots & 0 & 1 & -1 \\ \hline
			 0 & \cdots & 0 & -\gamma_1 & \gamma_1 \\
			 \vdots & \ddots & \vdots & \vdots & \vdots \\
			 0 & \cdots & 0 & -\gamma_p & \gamma_p
		\end{array}
	\right]
	\quad
	\text{and}
	\quad
	N^*
	=
	\left[
		\begin{array}{ccc|c}
			 &  &  & \alpha'_1 - \alpha_1 \\
			 & N_c^* &  & \vdots \\
			 &  &  & \alpha'_n - \alpha_n
		\end{array}
	\right]\,,
\]
for some $n \times (m-1)$ matrix $N_c^*$, where $n$ is the number of nonintermediate species, $m - 1$ is the number of reactions not involving $Y$, and where we write $y = \alpha_1S_1 + \cdots + \alpha_nS_n + \gamma_1E_1 + \cdots \gamma_pE_p$ and $y' = \alpha'_1S_1 + \cdots + \alpha'_nS_n + \gamma_1E_1 + \cdots + \gamma_pE_p$. Thus,
\begin{equation}\label{eq:ker_iiad}
	\kernel N = \{(v_1, \ldots, v_{m-1}, v_m, v_m) \in \r^{m + 1}\,|\ (v_1, \ldots, v_m) \in \kernel N^*\}\,.	
\end{equation}

In case $p = 0$, the argument is basically the same. The only difference is that $N$ does not have the $p$ bottom-most rows corresponding to the catalysts. The relationship between $\ker N$ and $\ker N^*$ is also given by (\ref{eq:ker_iiad}).

\subsubsection*{Relationship between $K$ and $K^*$}

We order the reactions $R_1, \ldots, R_m, R_{m + 1}$ of $G$ and $R^*_1, \ldots, R^*_m$ of $G^*$, so that $R_j$ and $R^*_j$ are identified for $j = 1, \ldots, m - 1$, and $R_m = y \ \text{---} \ Y$, $R_{m + 1} = Y \ \text{---} \ y'$, and $R^*_m = y \ \text{---} \ y'$. Note that the R-graph of $G^*$ could be obtained from the R-graph of $G$ by simply collapsing the edge $\{R_m, R_{m+1}\}$ in $G_R$ into the vertex $R^*_m$ in $G^*_R$ (refer to Figures \ref{fig:directed_sr_proof_b} and \ref{fig:directed_sr_proof_c}). Indeed, we have
	\begin{equation}\label{eq:labels1}
		L(\{R_m, R_{m+1}\}) = 1\,,
	\end{equation}
	and, in view of the positive loop property, $L(\{R_j, R_m\}) = L(\{R_j, R_{m + 1}\})$ for any $j \in [m - 1]$ such that $\{R_j, R_m\}, \{R_j, R_{m + 1}\} \in E_R$. Thus,
	\begin{equation}\label{eq:labels2}
		L(\{R_i, R_j\}) = L^*(\{R^*_i, R^*_j\})\,,
		\quad
		\forall i, j \in [m] \ \colon \ \{R_i, R_j\} \in E_R\,,
	\end{equation}
	and
	\begin{equation}\label{eq:labels3}
		L(\{R_j, R_{m+1}\}) = L^*(\{R^*_j, R^*_m\})\,,
		\quad
		\forall j \in [m - 1] \ \colon \ \{R_j, R_{m+1}\} \in E_R\,.
	\end{equation}
	
	Let $\sigma = (\sigma_1, \ldots, \sigma_m, \sigma_{m+1})$ and $\sigma^* = (\sigma^*_1, \ldots, \sigma^*_m)$ be the sign patterns of the  {orthants} $K$ and $K^*$ constructed via Remark \ref{rem:orthant_cone} for $G$ and $G^*$, respectively. It follows from (\ref{eq:sigma_def_b}) and (\ref{eq:labels1})--(\ref{eq:labels3}) that $\sigma_m = \sigma_{m + 1}$ and $\sigma^*_j = \sigma_j$, $j = 1, \ldots, m - 1$.

\subsubsection*{Proof of Theorem \ref{thm:main2}}

We may summarize the discussion above as a lemma.

\begin{lemma}\label{lemm:summary}
	Assume the same hypotheses as in Theorem \ref{thm:main2}. Then
	\[
		\kernel N = \{(v_1, \ldots, v_{m-1}, v_m, v_m) \in \r^{m+1}\,|\ (v_1, \ldots, v_m) \in \kernel N^*\}
	\]
	and the sign pattern $\sigma$ of $K$ is given by
	\[
		\sigma = (\sigma^*_1, \ldots, \sigma^*_m, \sigma^*_m)\,,
	\]
	where $(\sigma^*_1, \ldots, \sigma^*_m) = \sigma^*$ is the sign pattern of $K^*$.
\end{lemma}

It follows from the lemma that
\[
	(v_1, \ldots, v_m) \in \ker N^* \cap K^*
	\quad
	\Leftrightarrow
	\quad
	(v_1, \ldots, v_m, v_m) \in \ker N \cap K\,,
\]
and, moreover,
\[
	(v_1, \ldots, v_m) \in \interior K^*
	\quad
	\Leftrightarrow
	\quad
	(v_1, \ldots, v_m, v_m) \in \interior K\,.
\]
This establishes Theorem \ref{thm:main2}.

\paragraph{Acknowledgements. }
	Elisenda Feliu, Michael Marcondes de Freitas and Carsten Wiuf acknowledge funding from the Danish Research Council of Independent Research.  {We would also like to thank Anne Shiu and Mitchell Eithun for their careful reading of an earlier version of this paper and valuable comments.}


\end{document}